%% file: camera-uai2026-version.tex
\documentclass[accepted]{uai2026} 
                        
\usepackage[american]{babel}

\usepackage{natbib} 
\usepackage[utf8]{inputenc}
\usepackage[T1]{fontenc}
\usepackage{graphicx}
\usepackage{amsmath,amsfonts,amssymb,amsthm}
\usepackage{mathtools}
\usepackage{tikz}
\usepackage{mathrsfs}
\usepackage{xcolor}
\usepackage{hyperref}
\usepackage{cleveref}
\usepackage{booktabs}
\usepackage{bbold}
\usepackage{bm}
\usepackage{color}
\usepackage[font=footnotesize]{caption}
\usepackage{enumitem}
\usepackage{empheq}
\usepackage{framed}
\usepackage{soul}
\usepackage{dirtytalk}
\usepackage{xspace}
\usepackage{nicefrac}
\usepackage{autonum}
\usepackage[toc,page]{appendix}
\usepackage{subcaption}
\usepackage{cancel}
 
\input{macros}

\title{Exact Permutation Recovery Under Unknown Scalar Affine Transformation}

\author[1, 2]{\href{mailto:tigran.galstyan@centralesupelec.fr}{Tigran Galstyan}}
\author[1, 2]{Avetik Karagulyan}
\author[1]{Arshak Minasyan}

\affil[1]{%
    Laboratoire des Signaux et Syst\`emes (L2S)\\
    CentraleSup\'elec, Universit\'e Paris-Saclay\\
    Gif-sur-Yvette, France
}
\affil[2]{%
    Centre National de la Recherche Scientifique (CNRS)\\
    Gif-sur-Yvette, France
}

\begin{document}
\maketitle

\begin{abstract}
We study the problem of matching two sets of noisy feature vectors when underlying true features are related by an unknown scalar affine transformation.
Our method comprises two primary steps. 
First, we standardize the feature vectors to estimate the unknown scalar affine transformation. 
Subsequently, we estimate the permutation by minimizing the Least Sum of Logarithms (LSL) between two sets of observations using the estimated transformation.

Our main result shows that the unknown permutation can be perfectly recovered given that the minimal separation distance of true feature vectors scales as $\sqrt{\rhosigma} \vee (d\log n)^{1/4} \vee \sqrt{\log n}$, where $d$ is the ambient dimension, $n$ is the sample size, and $\rhosigma$ is the maximal ratio of noise magnitudes.
Interestingly, the obtained rate, under mild heteroscedasticity, coincides with that of the non-affine setting.
We additionally demonstrate that there exist configurations requiring a larger minimal separation distance for perfect recovery. 
The latter makes the matching problem more challenging from minimax perspective compared to the non-affine setting.

Consequently, we show that in the problem of feature matching, standardizing the data implicitly estimates the scalar affine parameters. 
As part of our analysis, we prove non-asymptotic concentration bounds for the affine parameter estimators in the presence of heterogeneous noise magnitudes.

\end{abstract}

\input{uai-sections/Introduction}
\input{uai-sections/Homogeneous}
\input{uai-sections/framework}
\input{uai-sections/experiments}
\input{uai-sections/Conclusion}


\begin{acknowledgements}
    We would like to thank Arnak Dalalyan and Yann Issartel for helpful discussions. 
\end{acknowledgements}
\bibliography{uai2026}

\newpage

\onecolumn

\title{Exact Permutation Recovery Under Unknown Scalar Affine Transformation\\(Supplementary Material)}
\maketitle

\appendix
\input{uai-sections/Appendix}

\end{document}

%% file: macros.tex
\newtheorem{definition}{Definition}

\newtheorem{theorem}{Theorem}
\newtheorem{lemma}{Lemma}

\newtheorem{proposition}{Proposition}

\theoremstyle{remark}

\definecolor{midnightblue}{HTML}{0059b3}
\definecolor{noonblue}{HTML}{e5eef7}
\definecolor{chromered}{HTML}{f14233}
\definecolor{olivedrab}{HTML}{6b8e23}
\definecolor{darkmidnightblue}{HTML}{154c84}
\definecolor{textblue}{HTML}{11406e}
\definecolor{darkgreen}{HTML}{0e6029}
\definecolor{mainblue}{HTML}{154c84} 

\newcommand{\eqtext}[1]{\quad\text{#1}\quad}

\newcommand{\brr}[1]{{\left( #1 \right)}} 

\newcommand{\brs}[1]{{\left[ #1 \right]}} 

\newcommand{\brc}[1]{{\left\{ #1 \right\}}} 

\newcommand{\norm}[1]{{\left\lVert #1 \right\rVert}}

\renewcommand{\leq}{\leqslant}
\renewcommand{\geq}{\geqslant}

\newcommand{\cA}{\mathcal{A}}

\newcommand{\cE}{\mathcal{E}}

\newcommand{\RR}{\mathbb{R}}
\newcommand{\R}{\mathbb{R}}

\newcommand{\ba}{\boldsymbol{a}}

\newcommand{\bV}{\boldsymbol{V}}

\newcommand{\bX}{\boldsymbol{X}}
\newcommand{\bY}{\boldsymbol{Y}}

\newcommand{\bZ}{\boldsymbol{Z}}
\def\bZdiese{\boldsymbol{Z}^\text{\tt\#}}

\newcommand{\btheta}{\boldsymbol{\theta}}

\newcommand{\kappaX}{{\kappa}({\bX})}
\newcommand{\kappaZ}{{\kappa}({\bZ})}
\newcommand{\rhosigma}{{\rho}_\sigma}

\newcommand{\bmu}{\boldsymbol{\mu}}

\newcommand{\bxi}{\boldsymbol{\xi}}

\newcommand{\bsigma}{\boldsymbol{\sigma}}

\newcommand{\id}{\textsf{id}}

\newcommand{\bTheta}{\boldsymbol{\Theta}}

\newcommand{\Prob}{\mathbf{P}}
\renewcommand{\phi}{\varphi}
\renewcommand{\epsilon}{\varepsilon}

\newcommand{\etau}{\varepsilon_{\tau}}

\usepackage{mathtools}

\def\tilde{\widetilde}
\def\hat{\widehat}

\newcommand{\Exp}{\mathbf{E}}
\newcommand{\Var}{\mathbf{var}}

\newcommand{\argmin}{\mathop{\mathrm{arg\,min}}}
\newcommand{\argmax}{\mathop{\mathrm{arg\,max}}}
\newcommand{\eqdef}{\mathrel{\mathop:} = } 

\def\trace{\mathbf{tr}}

\def\tilde{\widetilde}
\def\hat{\widehat}

\newcommand{\ie}{{\em i.e.,~}}

\def\t{\theta}

\def\bX{\boldsymbol{X}}
\def\bXdiese{\boldsymbol{X}^{\text{\tt\#}}}
\def\bXtildediese{\boldsymbol{\tilde{X}}^{\text{\tt\#}}}

\def\Xibardiese{{\overline{\xi}}^{\text{\tt\#}}}
\def\Xbardiese{\bar{{X}}^{\text{\tt\#}}}

\def\Xtildediese{\tilde{{X}}^{\text{\tt\#}}}

\def\Xbar{\bar{{X}}}
\def\Xibar{\overline{\xi}}

\def\ba{\mathbf{a}}

\def\btdiese{\boldsymbol{\theta}^\texttt{\#}}

\def\Xdiese{X^{\text{\tt\#}}}

\def\Zdiese{Z^{\text{\tt\#}}}

\def\sdiese{\sigma^{\text{\tt\#}}}
\def\sdiesesq{\sigma^{\text{\tt\#}2}}

\def\xidiese{\xi^\text{\tt\#}}

\def\bsdiese{\bsigma^\text{\tt\#}}

\def\btdiese{\btheta^\text{\tt\#}}
\def\tdiese{\theta^\text{\tt\#}}
\def\Sn{{\mathfrak{S}_n}}

\def\taustar{\tau^*{}}
\def\betastar{\beta^*{}}
\def\tauhat{\hat\tau}
\def\betahat{\hat\beta}

%% file: uai-sections/Introduction.tex

\begin{figure*}[ht]
    \centering
    \includegraphics[width=0.99\textwidth, height=120px]{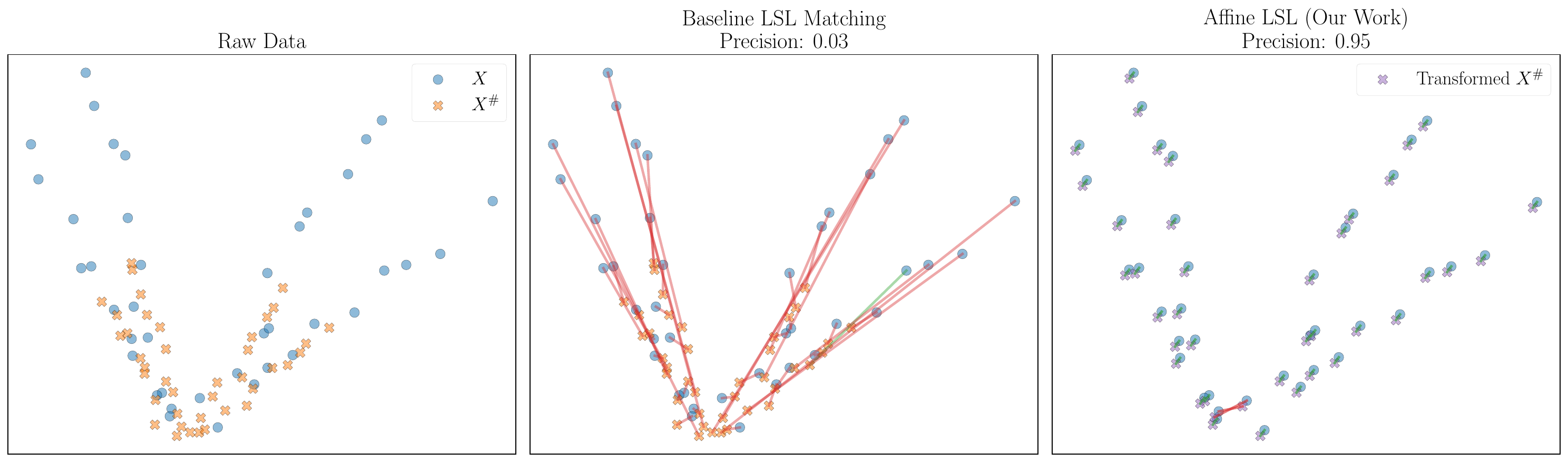}
    \caption{\textbf{Step-by-step affine matching process} 
    \textit{(Left)} PCA projection of the raw data generated according the model \eqref{eq:model}, showing the affine gap between $\bX$ and  $\bXdiese$. 
    \textit{(Middle)} The failure of standard LSL, which ignores the transformation and results in a precision of only 0.03. 
    \textit{(Right)} Our Affine LSL method estimates the transformation to align the $\Xdiese$ (purple crosses) with $X$, achieving 0.95 precision.}
    \label{fig:matching_viz}
\end{figure*}

\section{Introduction} \label{sec:introduction}
The problem of recovering the latent correspondence between two sets of noisy observations is a fundamental problem in high-dimensional statistics and data analysis, with applications ranging from point cloud registration in computer vision to sequence alignment in bioinformatics \citep{liumatcher, rdmnet, barkas2019joint}.
While the statistical limits of the matching problem have been extensively studied under simple additive noise models and the presence of outliers \citep{bakirtas, minasyan2023matching}, practical scenarios often present more complex challenges.
Specifically, the observations can be subject to an unknown affine transformation, caused by factors such as different coordinate systems, sensor calibration \citep{zhu2023mario}, measurement scales, or other environmental conditions \citep{cai2023doppelgangers}.
This work makes a step towards addressing this gap by modeling the unknown scale and translation parameters alongside the unknown permutation and analyzing the statistical limits of distinguishing the true matching from spurious alignments.

The general formulation of the matching problem is as follows. Suppose we observe two independent sequences of $d$-dimensional random vectors $\bX = \{X_i\}_{i=1}^n$ and $\bXdiese = \{\Xdiese_i\}_{i=1}^n$. In addition, we suppose that under some unknown permutation $\pi^* \in \Sn$, $\Xdiese_i, X_{\pi^*(i)} \stackrel{\textup{ind.}}{\sim} {\sf P}_i$ holds for any $i\in[n] := \{1, \dots, n\}$. The objective of the statistical matching problem is to recover the latent feature matching map $\pi^*$ from the noisy observations $\bX$ and $\bXdiese$. 

This paper studies the statistical matching problem under unknown scalar affine transformation. 
In particular, we observe two sequences of $d$-dimensional random vectors 
$\bX = \brc{X_i}_{i=1}^n$ and $\bXdiese = \{\Xdiese_j\}_{j=1}^n$ sampled according to the following model:
\begin{align} \label{eq:model}
    X_i = \theta_i + \sigma_i \xi_i, \quad \Xdiese_j = \tdiese_j + \sdiese_j \xidiese_j, \quad \forall\, i,j \in [n].
\end{align}
Here, $\btheta := \{\theta_i\}_{i=1}^n$ and $\btdiese := \{\tdiese_j\}_{j=1}^n$ are $d$-dimensional unknown feature vectors, while $\bsigma := (\sigma_1, \ldots, \sigma_n)^\top$ and $\bsdiese := (\sdiese_1, \ldots, \sdiese_n)^\top$ are unknown noise magnitudes. 
The sequences $\xi_1, \ldots, \xi_n $, and $\xidiese_1, \ldots, \xidiese_n$ are assumed to be  i.i.d. standard Gaussian vectors.
The key structural assumption is that there exists an unknown permutation $\pi^* \in \Sn$, an unknown scaling factor $\tau^* > 0$, and an unknown translation vector $\beta^* \in \mathbb{R}^d$ such that
\begin{align}\label{eq:affine-assumptions}
\theta_i = \tau^* \tdiese_{\pi^*(i)} + \beta^*, \quad \sigma_i = \tau^* \sdiese_{\pi^*(i)}, \quad \forall\, i \in [n].
\end{align}
Our goal is to jointly estimate  $\pi^*,\tau^*$, and $\beta^*$ from the noisy observations $\bX$ and $\bXdiese$.

\paragraph{Known affine transformation.}
When the parameters $\tau^*$ and $\beta^*$ are known, the problem can be directly solved using existing methods. Specifically, observe that the sets $\bX$ and $\bY^* := \brc{\tau^* \Xdiese_i + \beta^*}_{i\in[n]}$ share the same distribution up to the permutation $\pi^*$. Thus, estimating $\pi^*$ reduces to a standard permutation estimation problem, which was studied in \citet{collier2016minimax}. To briefly present their results, we define the minimal separation distance. 

\begin{definition}
    For independent random vectors $\bV = \brc{V_i}_{i\in[n]}$
    the separation distance $\kappa$ is a function defined as
    \begin{align}
        \kappa(\bV)
        \eqdef \! \min_{j\neq i} \big\|({\Var[V_i] + \Var[V_j]})^{-\frac{1}{2}}(\Exp[V_i] - \Exp[V_j])\big\|_2.
    \end{align}
\end{definition}

In the non-affine setting, and in the presence of Gaussian isotropic noise, \citet[Theorem 3]{collier2016minimax} assert that if
\begin{align}
    \kappaX \gtrsim \brr{d\log (n^2/\delta)}^{1/4} \vee \sqrt{\log (n^2/\delta)}, 
\end{align}
then the unknown permutation $\pi^*$ can be perfectly recovered with probability at least $1-\delta$. Moreover, \citet[Theorem 4]{galstyan2021optimal} assert that this rate is minimax optimal up to constant factors, providing the corresponding lower bound on the minimal separation distance.

\paragraph{Estimation procedures of unknown parameters.} 
In general, one cannot know in advance whether there is an underlying affine transformation applied to one of the datasets or not. Thus, a naive approach of directly matching the observations $\bX$ and $\bXdiese$ may result in failure to recover the unknown permutation $\pi^*$, while the minimal separation distance $\kappaX$ being arbitrarily large. Importantly, this does not contradict the previously established minimax rates, as their analysis assumes the absence of any underlying affine transformation. In other words, under model misspecification of permutation estimation the minimax rate on the minimal separation distance presented above is no longer valid.

In what follows, we propose a sequential estimation procedure: first, we estimate the unknown parameters $\tau^*$ and $\beta^*$, and then use these estimators to design an estimator for the unknown permutation $\pi^*$.
In particular, we propose the following estimator for $\taustar$
\begin{align}\label{eq:intro-tau-hat}
    \hat{\tau}_n 
    :=  \brr{\frac{\sum_{i=1}^n \| X_i - \bar{X}_n\|_2^2}{\sum_{i=1}^n \| \Xdiese_i - \Xbardiese_n\|_2^2}}^{1/2},
\end{align}
where $\Xbar_n$ and $\Xbardiese_n$ are the sample mean vectors of the sets $\bX$ and $\bXdiese$, respectively. Notice that $\hat{\tau}_n$ can be viewed as the ratio of sample standard deviations of $\bX$ and $\bXdiese$. 
The translation vector $\betastar$ can be estimated by $\hat{\beta}_n := \Xbar_n - \hat{\tau}_n\Xbardiese_n.$

Using these estimators, we show that the problem of recovering the underlying permutation $\pi^*$ reduces to matching the centralized observations $\bZ$ and ${\bZdiese}$ (scaled by $\tauhat_n$), where $\bZ = \brc{X_i - \bar{X}_n}_{i=1}^n$ and $\bZdiese = \brc{\Xdiese_i - \Xbardiese_n}_{i=1}^n$. 
Thus, in the presence of an affine transformation, we show that the quantity that plays role for recovering $\pi^*$ is indeed $\kappaZ$. 
In the non-affine setting when the permutation recovery is guaranteed with high probability, we prove that if the minimal separation distance satisfies $\kappaZ \gtrsim \kappaX \vee \sqrt{\rhosigma}$
then the unknown permutation $\pi^*$ can be perfectly recovered with high probability under unknown scalar affine transformation.
Here, $\rhosigma := \max_{i, j} (\sigma_j^2 / \sigma_i^2)$ is the maximal ratio of noise magnitudes. 
This quantity arises both in the estimation error of $\tauhat_n$ and in the condition on $\kappaZ$. 
We discuss the dependence on $\rhosigma$ at the end of this section.

The estimation procedure for $\pi^*$ is rooted in maximizing the profile likelihood function of our model \eqref{eq:model}, while using permutation invariant estimators $\hat{\tau}_n$ and $\betahat_n$.
The profile MLE approach is common in the literature and has been used in related papers, such as \citep{han2025covariance, galstyan2021optimal, collier2016minimax}.
In our setting the profile MLE approach leads to the following estimator for $\pi^*$, denoted by $\hat\pi^{\rm LSL}_n$:
\begin{align}\label{eq:pi-tau-mle}
    \hat\pi^{\rm LSL}_n \in \argmin_{\pi \in \Sn} \sum_{i=1}^n \log \brr{{\|Z_i- \hat{\tau}_n\Zdiese_{\pi(i)}\|_2^2}}.
\end{align} 
We highlight that our choice of $\hat{\tau}_n$ and $\betahat_n$ implicitly apply standardization to the observations $\bX$ and $\bXdiese$. Thus, when applying standardization, in the problem of statistical matching, we estimate the unknown transformation of the form $\tau I_d + \beta$, using the aforementioned estimators $\hat{\tau}_n$ and $\betahat_n$. We defer the detailed discussion of the standardization to \Cref{sec:main} and \Cref{sec:framework}. On the other hand, these estimators are closely connected to the profile MLE in the  homogeneous setting. 
Overall, the contributions of this paper are two-fold: estimation of the scalar affine parameters, and perfect recovery of the underlying permutation. The findings of this paper are briefly described below:
\paragraph{1.} Our main result, presented in \Cref{thm:pi_recovery}, establishes an upper bound on the minimal separation distance of the observations $\bZ$ so that with high probability the Least Sum of Logarithms estimator $\hat\pi^{\rm LSL}_n$ from \eqref{eq:pi-tau-mle} perfectly recovers the unknown permutation $\pi^*$. Specifically, in the high-dimensional regime $d \ge c \log n$ for some absolute constant $c$, we show that if
    \begin{align}
        \kappaZ \gtrsim \brr{d\log (n/\delta)}^{1/4} \vee \sqrt{\rhosigma\log(1/\delta)}, 
    \end{align}
    then our estimator exactly recovers $\pi^*$ with probability at least $1-\delta$. We refer to the case of \textit{mild heteroscedasticity} when $\rhosigma$ is of at most of order $\brr{d\log n}^{1/2} \vee \log n$. Interestingly, in this case, centralizing the original observations yields a minimax rate for separation distance that is identical to that of non-affine setting, implying its optimality. Moreover, when the maximal ratio of noise magnitudes $\rhosigma \gg \brr{d\log n}^{1/2} \vee \log n $ then the perfect permutation recovery requires the separation distance to scale proportionally to $\sqrt{\rhosigma}$. 

\paragraph{2.} The estimation error of the scaling factor $\taustar$ is presented in \Cref{thm:tau-affine}, where we establish a high-probability concentration bound of $\hat{\tau}_n$ around $\taustar$: under the model \eqref{eq:model}, the ratio $\hat{\tau}_n/\taustar$ concentrates around $1$. Up to logarithmic factors, the relative error $|\hat{\tau}_n^2/\taustar^2 - 1|$ scales as $\alpha_\sigma/\sqrt{\lambda^2 + d}$, where $\lambda$ is the \textit{distortion-to-noise ratio} and $\alpha_\sigma := \sigma_{\max}/\|\bsigma\|_2 \in [n^{-1/2}, 1]$ is the \textit{relative noise magnitude}, both defined in \Cref{sec:main}.

\paragraph{Notation.}\label{sec:notation}
    We use $[n] = \{1, \ldots, n\}$ to denote the set of integers from $1$ to $n$. 
    The symmetric group of all permutations on $[n]$ is denoted by $\Sn$. 
    For a vector ${v} = (v_1, \ldots, v_n)^\top \in \mathbb{R}^n$, we write $\|{v}\|_p = (\sum_{i=1}^n |v_i|^p)^{1/p}$ for the $\ell_p$ norm. 
    If the order is not explicitly written, then it is the Euclidean $\ell_2$-norm.
    The notation $\stackrel{d}{=}$ denotes equality in distribution. We denote by $\chi^2_d$ the chi-squared distribution with $d$ degrees of freedom. We write $a \wedge b = \min(a, b)$ and $a \vee b = \max(a, b)$.  We use the following notation for empirical averages $\bar{\theta}_n, \bar{{\t}}^{\text{\tt\#}}_n$ of the true features $\btheta$ and $\btdiese$, respectively. We use $\gtrsim$, $\lesssim$ for inequalities up to numerical constants.

\section{Related Work}

The problem of map recovery through matching has been widely studied in various settings, such as the geometric planted matching problem \citep{kunisky2022strong}, feature alignment \citep{han2025covariance}, and graph alignment \citep{ganassali2020tree, wang2022random, hall2023partial}. Some variation of robust to outliers versions of feature matching problem are studied in \citep{galstyan2021optimal, minasyan2023matching}.

Additionally, permutation estimation and related problems have recently been investigated in the context of statistical seriation \citep{flammarion2019optimal, issartel2024minimax}, clustering \citep{davies1979cluster, giraud2019partial}, noisy sorting \citep{mao2018minimax}, crowd labeling \citep{ShahBW16a}, and variable selection \citep{Ndaoud, ComDal1}.

The problem of permutation recovery has also been in a slightly different setting of unlinked or unmatched regression. In the latter setting, the response variables and their corresponding features are permuted. While the permutation recovery is not the main objective, it remains a crucial part of the problem. The case of linear regression was studies in originating in \citep{pananjady2017linear} and \citep{pananjady2017denoising}. This line of work has recently gained attention \citep{balabdaoui2021unlinked, azadkia2024linear, slawski2024permuted}. Inference from the anonymous and/or privatized data falls inside the framework of unlinked regression. 

The separation number plays a crucial role in statistical matching. 
However, its applications extend far beyond this area. 
In hypothesis testing, the minimal separation distance has emerged as the standard measure for distinguishing between null and alternative hypotheses, as illustrated in the seminal paper \citep{Ingster82} and the monograph \citep{JudNem}. 
This concept has since found numerous applications in the recent literature \citep{Wolfer, blanchard, Yuting}.

The problem of matching under unknown affine transformations has been studied in natural language processing by \citet{grave2019unsupervised,pumir2021generalized}. Specifically, \citet{grave2019unsupervised} show that simultaneously estimating an unknown permutation and affine transformation is non-convex. To address this, they propose a convex relaxation with alternating minimization of the objective function. 
However, to our knowledge, the statistical limitations from the minimax point of view for the matching recovery were not studied.

%% file: uai-sections/Homogeneous.tex

\section{Main theoretical results}\label{sec:main}

This section contains the main theoretical findings of the paper. We focus on the case where the affine parameters are unknown. We first present a 
finite-sample high-probability concentration bound for the estimator $\hat{\tau}_n$ around the true parameter $\taustar$. As for the translation parameter $\betastar$, we do not assess the estimation error of $\hat{\beta}_n$ explicitly. 
Instead, it is done implicitly when considering the centralized observations $\bZ$ and $\bZdiese$. 
Given that the norm of the empirical averages of the features $\|\bar{\theta}_n \|_2$ is large, the consistent estimation of $\betastar$ is not possible since the estimation error of $\taustar$ propagates. 
At the same time, the estimator $\hat{\tau}_n$ cannot compensate for this effect, since it is invariant to translations. However, we show that this does not affect permutation recovery.
We then present the main result of this paper, that is, we present the condition on the minimal separation distance between the centralized observations $\bZ$ and $\bZdiese$ such that $\hat{\pi}_n^{\rm {LSL}}$ perfectly recovers the unknown permutation $\pi^*$ with high probability.

\subsection{Performance analysis of affine parameter estimators}
Recall that centralized observations $\bZ$ and $\bZdiese$ are defined as $Z_i = X_i - \bar{X}_n$, $\Zdiese_i = \Xdiese_i - \Xbardiese_n$ for all $i \in [n]$. Notice that for any pairs $(i, j) \in [n]^2$ the random vectors $Z_i$ and $\Zdiese_j$ are independent. Moreover, for each $i \in [n]$, we have $Z_i \stackrel{d}{=} \tau^*\Zdiese_{\pi^*(i)}  \sim \mathcal{N}(\mu_i, s_i^2 I_d)$ with $\mu_i$ and $s_i$ defined as
\begin{align}\label{eq:def-mu-xibar}
      \mu_i \eqdef \theta_i - \bar{\theta}_n, \qquad s_i^2 := \frac{n-2}{n} \sigma_i^2 + \frac{1}{n^2} \| \bsigma\|_2^2.
\end{align}
We first establish two-sided high-probability bounds on the sum of squared $\ell_2$-norms $\sum_{i=1}^n \|Z_i\|_2^2$ of the centralized observations $Z_1, \ldots, Z_n$. 
This result will then be used to prove that the estimator $\hat{\tau}_n$, defined in \eqref{eq:intro-tau-hat}, accurately estimates the unknown parameter $\tau^*$.

\begin{proposition}\label{prop:mean-diff}
    Let $Z_1, \ldots, Z_n$ be $d$-dimensional Gaussian vectors with mean vectors $\mu_1, \dots, \mu_n$ and isotropic covariance matrices $s_1^2 I_d, \dots, s_n^2 I_d$, respectively, according to \eqref{eq:def-mu-xibar}. Denote $\sigma_{\max} := \max_{i \in [n]} \sigma_i$. Then, for any $\delta \in (0, 1)$, the following inequality holds with probability at least $1-\delta$:
    \begin{align}
    \sum_{i=1}^n \|Z_i\|_2^2 &- \|\bmu\|_2^2 - d\|\bsigma\|_2^2
\le2\sigma_{\max}\|\bmu\|_2\sqrt{2\log(2/\delta)} \\
     &+ 2\|\bsigma\|_4^2\sqrt{d\log(2/\delta)} + 2\sigma_{\max}^2 \log(2/\delta).
    \end{align}
    Similarly, with probability at least $1-\delta$, we have
    \begin{align}
        \sum_{i=1}^n \|Z_i\|_2^2 &- \|\bmu\|_2^2 - d\|\bsigma\|_2^2
\ge - 2\sigma_{\max}\|\bmu\|_2\sqrt{2\log(2/\delta)} \\
    &- 2\|\bsigma\|_4^2\sqrt{d\log(2/\delta)} - \frac{d}{n}\|\bsigma\|_2^2.
    \end{align}
\end{proposition}

The proof is deferred to \Cref{sub:proof_of_prop_mean_diff}. 
\Cref{prop:mean-diff} is related to the Hanson-Wright inequality \citep{rudelson2013hanson}, which provides a general framework for bounding the quadratic form ${X}^\top {\sf A} {X}$, where ${X} = \brr{X_1, \ldots, X_d} \in \R^{d}$ is a $d$-dimensional sub-Gaussian vector with independent coordinates and ${\sf A} \in \R^{d \times d}$ is an arbitrary matrix.
In contrast, our setting involves correlated random vectors $Z_i$, arising from the centering step that projects the observations onto the subspace orthogonal to the all-ones direction. Despite this dependence, we still obtain a comparable concentration bound: as detailed in \Cref{sub:proof_of_prop_mean_diff}, the proof recasts the noise contribution to $\sum_{i=1}^n \|Z_i\|_2^2$ as a quadratic form in an isotropic Gaussian vector, for which a Hanson-Wright-type bound is available. Our result corresponds to the special case ${\sf A} = I_d$; we believe \Cref{prop:mean-diff} can be generalized to an arbitrary matrix ${\sf A}$, but we do not pursue this direction here. Additionally, \Cref{prop:mean-diff} can be recast as a concentration inequality for the squared Frobenius norm of the Gaussian matrix $(Z_1, \dots, Z_n) \in \RR^{d\times n}$ around its mean $\Exp\brs{\|\bZ\|_{\operatorname{Fr}}^2} = \|\bmu\|_2^2 + d\tfrac{n-1}{n}\|\bsigma\|_2^2$, which may be of independent interest. Formally, each of the following inequalities holds with probability at least $1-\delta$:
\begin{align}
    \|\bZ\|_{\operatorname{Fr}}^2 - \Exp \brs{\|\bZ\|_{\operatorname{Fr}}^2} \le \bar{\sf r}_{\textup{ub}}(\bsigma, d, \delta)
\end{align}
and 
\begin{align}
    \|\bZ\|_{\operatorname{Fr}}^2 - \Exp \brs{\|\bZ\|_{\operatorname{Fr}}^2} \ge \bar{\sf r}_{\textup{lb}}(\bsigma, d, \delta),
\end{align}
where $\bar{\sf r}_{\textup{ub}}(\bsigma, d, \delta)$ and $\bar{\sf r}_{\textup{lb}}(\bsigma, d, \delta)$ are the corresponding upper and lower bounds from \Cref{prop:mean-diff}, respectively. The next result proves that the estimator $\tauhat_n$ from \eqref{eq:intro-tau-hat} of the scaling factor concentrates around $\taustar$. The proof is postponed to \Cref{sub:proof_of_thm_tau_affine} and relies on \Cref{prop:mean-diff}.  

\begin{theorem}[Scaling factor estimation]\label{thm:tau-affine}
Consider the observations $\bX$ and $\bXdiese$ following the model in \eqref{eq:model}
and satisfying the affine conditions from \eqref{eq:affine-assumptions}.
Let $\tauhat_n$ be the estimator defined in \eqref{eq:intro-tau-hat}. 
Let $\alpha_\sigma := \sigma_{\max}/\|\bsigma\|_2$ be the relative noise magnitude,
and $\lambda = \|\bmu\|_2/\|\bsigma\|_2$ be the distortion-to-noise ratio. Fix $\delta \in [4e^{-d/(224\alpha_\sigma^2)}, 1)$ and
assume that $n \ge 8$.
Then, for any $\taustar > 0$, with probability at least $1 - 4\delta$, we have
\begin{equation}
    \Big| \frac{\hat{\tau}_n^2}{\taustar^2} - 1 \Big| \le 12\sqrt{\frac{\alpha_\sigma^2\log(4/\delta)}{\lambda^2 + d}} + \frac{2\alpha_\sigma^2\log(4/\delta)}{\lambda^2 + d} + \frac{2d}{n(\lambda^2 + d)}.
\end{equation}

\end{theorem}

Let us now comment on the estimation error of the scaling factor $\taustar$. The bound consists of several components. The first term, proportional to $\alpha_{\sigma}/ \sqrt{\lambda^2 + d}$, up to $\sqrt{\log(1/\delta)}$ corresponds to the variance of ${\tauhat_n}^2$ in the small-deviation regime. Similarly, in the large-deviation regime, the dependence scales as $\alpha_{\sigma}^2/(\lambda^2 + d)$ up to a logarithmic factor in the confidence level $\delta$, that is $\log(1/\delta)$. This division into small- and large-deviation regimes reflects the sub-exponential character of the distribution of $\tauhat_n$, see, for example, \citep[Section 2.8]{vershynin2018high}.
 The remainder of the upper bound of \Cref{thm:tau-affine} scales as $d/(n(\lambda^2+d))$, which corresponds to the bias of the estimator $\tauhat_n^2$. Indeed, it is positively biased, which can be seen by applying Jensen's inequality to the mapping $t \to 1/t$ and the independence of $\bZ$ and $\bZdiese$. It is also worth mentioning that the bias term vanishes as $n$ grows.

The parameter $\alpha_\sigma$ present in the deviation terms admits a natural interpretation as a noise-concentration coefficient: it measures how much of the total noise variance $\|\bsigma\|_2^2$ is carried by the  largest component $\sigma_{\max}^2$. The range of possible values of $\alpha_{\sigma}$ is on the closed interval $[n^{-1/2}, 1]$, with the lower endpoint $\alpha_\sigma = n^{-1/2}$ attained in the case of homogeneous noise, that is $\sigma_i = \sdiese_j = \sigma$ for any $(i,j) \in [n]^2$. Thus, in the homogeneous case the estimation rate scales as $1/\sqrt{n(\lambda^2+d)} \vee 1/(n(\lambda^2 + d))$, which matches the optimal parametric rate of estimating a scalar parameter given $2nd$ noisy observations. Additionally, if the distortion-to-noise ratio is large enough, \textit{e.g.,} $\lambda^2 \gtrsim n$ we get fast rates of order $1/n$. At the opposite extreme, when $\alpha_\sigma$ approaches $1$, the estimation rate scales as $1/\sqrt{\lambda^2 + d}$. This effectively means that without the knowledge of $\pi^*$, only $d$ coordinates carry information about $\taustar$. A simple example of such configuration is when $\sigma_1 = 1$ and $\sigma_2 = \ldots = \sigma_n = 1/n^2$ making $\alpha_{\sigma} \asymp 1/(1+n^{-1})$. Overall, the parameter $\alpha_{\sigma}$ describes the noise magnitude configuration and its effect on the estimation rate. In the homogeneous case, the estimation rate become faster, while in a single dominated noise configuration it deteriorates. It is also worth mentioning that the result of \Cref{thm:tau-affine} is generic and holds true for any configuration of true unavailable feature vectors $\btheta$ and noise magnitudes $\bsigma$. The underlying structure of the configuration of mean vectors and noise magnitudes can make the statistical problem of estimation of scaling factor easier (fast rates) or harder (slow rates). 

Moreover, the obtained high-probability bound due to its nature also provides us with an in-expectation bound. Indeed, integrating both sides of the bound from \Cref{thm:tau-affine} with respect to $\delta$, the proposed estimator $\hat{\tau}_n$ satisfies
\begin{align}
    \sup_{\taustar > 0} \frac{1}{\taustar{}{^2}}\Exp\brs{|\hat{\tau}_n^2 - \taustar{}^2|}
    &\lesssim \sqrt{\frac{\alpha_\sigma^2}{\lambda^2 + d}} \vee \frac{\alpha_\sigma^2}{\lambda^2 + d}.
\end{align}

The setting studied in this paper is closely related to the geometric planted matching problem studied in \citep{kunisky2022strong}. 
The latter assumes that $X_1, \dots, X_n \stackrel{\text{i.i.d.}}{\sim} \mathcal{N}(0, I_d)$ and $\Xdiese_i = X_{\pi^*(i)} + \xi_i$ for all $i\in[n]$, where $\xi_1, \dots, \xi_n$ are independent Gaussian vectors with zero mean and identical covariance matrices proportional to $I_d$, i.e., $\xi_1, \dots, \xi_n \stackrel{\text{i.i.d.}}{\sim} \mathcal{N}(0, \sigma^2 I_d)$.
They showed that, in the high-dimensional regime ($d \gg \log n$), the theoretical limit for exact recovery of the underlying matching $\pi^*$ is when $\sigma^2 \lesssim d/\log n$. The key difference with our model is that we assume that the observations $X_1, \ldots, X_n$ and $\Xdiese_1, \ldots, \Xdiese_n$ are independent. 
Since independence induces less structure in the data, perfect recovery in our model is not possible unless the minimal separation distance is large enough.
We quantify the order of the minimal separation distance in \Cref{thm:pi_recovery}, and show that in the case of mild heteroscedasticity it coincides with the minimax lower bound established in \cite[Theorem 4]{galstyan2021optimal} for known $\tau^*=1$ and $\beta^* = 0$. 
This lower bound directly extends to the setting with unknown parameters $\tau^*$ and $\beta^*$.

\subsection{Exact permutation recovery of the underlying permutation}
We are now ready to present our main result on the exact permutation recovery of the unknown permutation $\pi^*$. To ease the presentation the proof of the theorem is deferred to \Cref{sub:proof_of_thm_pi_recovery}. 

\begin{theorem}[Exact permutation recovery]\label{thm:pi_recovery}
Consider the centralized observations $\bZ$ and $\bZdiese$ of $\bX$ and $\bXdiese$, respectively. 
The latter sets follow the model \eqref{eq:model}, and satisfy the affine assumption from \eqref{eq:affine-assumptions}. Let $\tauhat_n$ be the estimator of $\tau^*$ as in \eqref{eq:intro-tau-hat},
and $\hat\pi^{\rm LSL}_n$ the LSL estimator from \eqref{eq:pi-tau-mle}.
Let $\rhosigma = \max_{i,j}\sigma_j^2/\sigma_i^2$ be the maximal ratio of noise magnitudes.
Fix $\delta \in [4e^{-d/(1024\alpha_\sigma^2)} ,\, 1 )$, and assume that $n \ge d$, and $n \ge 8$. 
Then, if the separation distance $\kappaZ$ is not smaller than
\begin{align}\label{eq:thm3-rate}
    5\big(d \log(12 n^3 \! / \delta)\big)^{1/4} + 17 \!\sqrt{\log(24 n^3 \!/ \delta)} + 52\!\sqrt{\rho_{\sigma}\log(4\!/\delta)},
\end{align}
the estimator $\hat\pi^{\rm LSL}_n$ coincides with $\pi^*$ with probability at least $1-4\delta$, \textit{i.e.} $\Prob\big(\forall i \in [n] ~:~ \hat\pi^{\rm LSL}_n(i) = \pi^*(i)\big) \ge 1 - 4\delta$.
\end{theorem}

The sufficient condition for exact recovery established above consists of three terms. Interestingly, these terms are of different nature. The first two terms depend only on sample size $n$, ambient dimension $d$, and confidence level $\delta$. These terms essentially come from controlling the supremum of Gaussian and $\chi^2_d$ random variables over a segment of constant length and maximum over $2n$ pairs. Unsurprisingly, compared to the non-affine setting, it results in the same minimax separation distance of order $(d\log n)^{1/4} \vee \sqrt{\log n}$. The last term, depends on the noise configuration, namely, the maximal ratio of noise magnitudes. This term comes from the application of \Cref{thm:tau-affine}, when we estimate the unknown scaling factor $\taustar$. Interestingly, in the bound of \Cref{thm:tau-affine} only the relative noise magnitude $\alpha_{\sigma} \in [n^{-1/2}, 1]$ appears. One would expect to have the same quantity appearing in the sufficient condition for the separation distance $\kappaZ$. However, careful inspection of the proof of \Cref{thm:pi_recovery} shows that for the matching problem the key quantity is the maximal ratio $\rho_{\sigma}$. Thus the entire dependence on noise heterogeneity is encoded in the ratio $\sigma_{\max}/\sigma_{\min} \equiv \sqrt{\rhosigma}$. Overall, we get that as long as $\kappaZ$ is at least of order 
\begin{align}
    \max\Big\{\big(d\log n)^{1/4}, ~ \sqrt{\log n}, ~\frac{\sigma_{\max}}{\sigma_{\min}}\Big\}
\end{align}
then the exact permutation recovery  under an unknown scalar affine transformation is possible with high probability. We further discuss the optimality of the result in the following paragraph.

\paragraph{Lower bounds.} 
\Cref{thm:pi_recovery} implies that the minimal separation distance between the centralized observations $\bZ$ and $\bZdiese$ that suffices for the perfect recovery of $\pi^*$ is at most of order $\big(d\log n\big)^{1/4} \vee \sqrt{\log n} \vee \sqrt{\rhosigma}$. 
A natural question is whether this rate is optimal in the minimax sense. Under mild heteroscedasticity, that is when $\rho_{\sigma} \lesssim \sqrt{d \log n} \vee \log n$, the rate is indeed minimax optimal, since the third term is dominated by the first two and the minimax optimality follows from \citep[Theorem~4]{galstyan2021optimal}. Under strong heteroscedasticity, that is when $\rho_{\sigma} \gg \sqrt{d \log n} \vee \log n$, there is a gap between the minimax lower bound (non-affine setting), and our upper bound. It remains open whether the lower bound for the non-affine setting is loose for our model, or our upper bound is not minimax optimal under strong heteroscedasticity. Additionally, the optimal dependence on $\rhosigma$, if any, remains open.

We perform numerical experiments to support the findings of \Cref{thm:pi_recovery}, see \Cref{sec:computational} for more detailed explanations of the experiments. In \Cref{fig:lsl_lower_bound}, we experimentally illustrate that the permutation recovery problem does become harder as $\rhosigma$ grows. That is, our algorithm starts to exactly recover the permutation when the separation distance becomes larger. For smaller values of the separation distance the estimator $\hat\pi^{\rm LSL}_n$ only partially recovers $\pi^*$. The problem of partial permutation recovery is out of scope of this paper.

\paragraph{Standardization cost.} 

In practice, it is typically unknown whether an underlying affine transformation exists in the permutation model.

If there is no affine transformation, that is $\taustar = 1$, $\betastar = 0$, and the maximal noise ratio $\rhosigma$ is large, our procedure becomes suboptimal due to error propagation in estimating $\taustar$. For large $\rhosigma$, the precision of $\hat{\tau}_n$ deteriorates, introducing additional error when matching $\bZ$ and $\hat{\tau}_n \bZdiese$. Consequently, the LSL method becomes more susceptible to noise, leading to failures in perfect permutation recovery. In such cases, data standardization should be avoided.

If an affine transformation is present and the heteroscedasticity is mild, that is $\rhosigma$ is of constant order relative to other parameters, our procedure is preferable. 
It achieves the minimax optimal rate for permutation recovery. According to \Cref{thm:tau-affine}, $\tauhat_n$ accurately estimates $\taustar$, and provided that $\kappaZ \gtrsim (d \log n)^{1/4} \vee \sqrt{\log n}$, the LSL method recovers $\pi^*$ perfectly with high probability. In this scenario, data standardization implicitly estimates the unknown transformation, enabling perfect recovery via LSL.

%% file: uai-sections/framework.tex

\section{Motivation and Discussion}\label{sec:framework}
In this section, we describe the approach based on the profile Maximum Likelihood Estimation (MLE) both in heterogeneous and homogeneous noise settings. In the latter case, we derive the estimators for parameters $\taustar$ and $\beta^*$ showing that they coincide with $\hat{\tau}_n$ and $\hat{\beta}_n$, respectively. Lastly, we draw the connection of these estimators to data standardization. 
\paragraph{Profile MLE.}\label{subsection:homogeneous_estimators}

Recall that our observations follow the statistical model \eqref{eq:model}, under the scalar affine transformation described in \eqref{eq:affine-assumptions}.
In the context of statistical matching, the main parameter of interest is the unknown permutation $\pi^*$, while the others are treated as nuisance parameters.
Below, we derive an estimator for $\pi^*$ based on maximizing the profile maximum likelihood function of the model \eqref{eq:model}. 
Let us define by $\bTheta$ the set of all nuisance parameters, that is $\{\btdiese,\bsdiese\}$. 
Notice that $\bX \cup \bXdiese$ is a set of $2n$ independent Gaussians. 
Thus, for any values of the parameters $\pi, \tau, \beta$, the negative logarithm of the full maximum likelihood function is up to some additive and multiplicative constants equal to 
\begin{align}
 \ell_n( \pi, \tau, \beta, &{\bTheta}; \{\bX, \bXdiese\}) 
 = \sum_{i=1}^n \Bigg(\frac{\big\|X_i - \tau\tdiese_{\pi(i)} - \beta\big\|_2^2}{\tau^2\sdiesesq_{\pi(i)}} 
 \\&+ \frac{\big\|\Xdiese_{\pi(i)} - \tdiese_{\pi(i)}\big\|_2^2}{\sdiesesq_{\pi(i)}} 
+ d\log\big(\tau\sdiesesq_{\pi(i)}\big) \Bigg), \label{eq:log-likelihood}
\end{align}
where we have already taken into account the scalar affine assumptions from \eqref{eq:affine-assumptions}, and got rid of the dependence on parameters $\btheta$ and $\bsigma$.
The set of parameters $\bTheta$ is unknown and our goal is to minimize the negative log-likelihood function $\ell_n$ with respect to them. Since, the function $\ell_n$ is not convex in $\bTheta$ we first minimize it with respect to $\btdiese$, then with respect to $\bsdiese$. 
Notice further that the log-likelihood is separable with respect to the variables $\tdiese_i$. Thus, for fixed  $\pi, \tau, \beta$ and $\bsdiese$, minimizing with respect to $\btdiese$ yields the following profile likelihood function up to an additive constant:
\begin{align}
     \ell_n( \pi, &\tau, \beta, \bsdiese; \{\bX, \bXdiese\}) \\
     &= \sum_{i=1}^n \Bigg( \frac{\big\|X_i - \tau\Xdiese_{\pi(i)} - \beta\big\|_2^2}{2\tau^2\sdiesesq_{\pi(i)}}  + d\log(\tau\sdiesesq_{\pi(i)}) \Bigg).
\end{align}
Further minimizing the expression from the previous display with respect to $\sdiese_{\pi(i)}$ for all $i$, under the condition that $\min_{i\in[n]} \sdiese_i := \sdiese_{\min} > 0$, we obtain the following expression for $\ell_n$:
\begin{align}\label{eq:likelihood-tau}
    \ell_n( \pi, \tau, \beta; \{\bX, \bXdiese\}) \!= \! \sum_{i=1}^n \log \brr{\frac{\|X_i  - \tau \Xdiese_{\pi(i)} - \beta\|_2^2}{\tau}}.
\end{align}
First, notice that the likelihood is non-convex with respect to $\tau$.
Second, for any $i \in [n]$ we have 
\begin{align}
    \lim_{\beta \rightarrow X_i - \tau \Xdiese_{\pi(i)}} \ell_n(\pi, \tau, \beta, \bX, \bXdiese) = -\infty.
\end{align}
Thus, there is no closed form solution for the minimization of $\ell_n$ in terms of $\beta$ and $\tau$. 
To circumvent this issue, we plug in the permutation invariant estimators $\hat{\tau}_n$ and $\hat{\beta}_n$ to estimate the unknown permutation $\pi^*$. Therefore, using the definition of centralized observations $\bZ$ and $\bZdiese$, the minimization of the quasi log-likelihood function yields the Least Sum of Logarithms estimator $\hat{\pi}_n^{\textup{LSL}}$, defined as per \eqref{eq:pi-tau-mle}.

The optimization problem that yields $\hat{\pi}_n^{\textup{LSL}}$ is a  computationally tractable linear assignment problem with a cost matrix $M_{ij} = \log(\| Z_i - \hat{\tau}_n \Zdiese_{j}\|_2^2)$ for all $(i, j) \in [n]^2$.
In particular, it can be efficiently solved using the Hungarian algorithm \citep{kuhn1955hungarian}. 
The idea of the Hungarian algorithm is based on the relaxation of the set of permutation and Birkhoff-von-Neumann theorem \citep{BudishCheKojimaMilgrom2009}. 
In particular, the permutation estimation reduces to 
\begin{align}
    \hat{\pi}_n^{\textup{LSL}}  \in \argmin_{P_{\pi} \in \mathcal{P}} \trace(M\cdot P_{\pi}),
\end{align}
where $\mathcal{P}$ is the set of all doubly stochastic matrices of size $n$.
We refer the reader to \citep{collier2016minimax, galstyan2021optimal} for more details on this relaxation and the linear assignment problem in the context of matching. To further justify the choice of $\hat{\tau}_n$ and $\hat{\beta}_n$, let us show their connection to the profile MLE estimators in the homogeneous noise setting.

\paragraph{Homogeneous setting.} The statistical problem of matching in the case of homogeneous noise case, where $\sdiese_i = \sdiese$ for all $i \in [n]$ does not pose the difficulties discussed earlier. To this end, we depart from the negative logarithm of the full maximum likelihood function considering that the noise magnitudes are equal. This simplifies the negative log-likelihood to:
\begin{align}
    \ell_n(\pi, \tau, \beta; \{\bX, \bXdiese\}) = \sum_{i=1}^n \big\|X_i - \tau\Xdiese_{\pi(i)} - \beta\big\|_2^2.
\end{align}
Now for any fixed $\tau = \tau_0$ the minimization with respect to $\beta$ yields the following permutation invariant estimator of $\beta^*$, denoted by ${\beta}_0 := \bar{X}_n - {\tau}_0\Xbardiese_n$. Then, plugging back in the expression for ${\beta}_0$ and using the notation for the centralized observations $\bZ$ and $\bZdiese$ we arrive at 
\begin{align}
    \ell_n(\pi, \tau_0; &\{\bX, \bXdiese\})=\sum_{i=1}^n \big\|Z_i - \tau_0\Zdiese_{\pi(i)}\big\|_2^2.\label{eq:likelihood-tau0}
\end{align}
By expanding the norms and removing all terms independent of $\pi$, we observe that the above function can be minimized without knowledge of $\tau_0$. 
Thus, minimizing \eqref{eq:likelihood-tau0} leads to the Least Sum of Squares (LSS) estimator, defined as
\begin{align}
    \hat{\pi}_n^{\textup{LSS}} 
    \in \argmax_{\pi \in \Sn} \sum_{i=1}^n Z_i^\top\Zdiese_{\pi(i)}.
    \label{eq:pi-lss}
\end{align}
Thus, in the homogeneous case, we showed that the knowledge of $\tau^*$ is not necessary to be able to estimate the unknown underlying permutation $\pi^*$. 
The maximization problem \eqref{eq:pi-lss} is again a linear assignment problem, and can be solved efficiently with time complexity $\mathcal{O}(n^3)$ using the Hungarian algorithm. Notice that the optimization problem \eqref{eq:likelihood-tau0} is quadratic with respect to $\tau$. Since the permutation recovery in the homogeneous case is independent of $\taustar$ one can estimate $\taustar$ after recovering the unknown permutation. The value that minimizes \eqref{eq:likelihood-tau0} reads as
\begin{align}
    \hat{\tau}_0 = \brr{\frac{\sum_{i=1}^n Z_i^\top\Zdiese_{\pi(i)}}{\sum_{i=1}^n \|\Zdiese_{\pi(i)}\|_2^2}}^{\frac{1}{2}}.
\end{align}
Now, we recognize that on the event $\Omega = \{\hat{\pi}_n^{\textup{LSS}} = \pi^*\}$ the denominator is close to the sum of squared $\ell_2$ norms of centralized observations $Z_1, \dots, Z_n$. Formally, let $\hat{\pi} \equiv \hat{\pi}_n^{\textup{LSS}}$, then conditionally on $\Omega$ we have
\begin{align}
     \sum_{i=1}^n \Exp\big[Z_i^\top\Zdiese_{\hat{\pi}(i)}\big] 
     = \sum_{i=1}^n \| \Exp[Z_i]\|_2^2,
\end{align} 
where the right hand side is permutation invariant. Therefore, this motivates us to use the estimator for $\tau^*$ that is defined as per \eqref{eq:intro-tau-hat}.
Even though this estimator does not maximize the profile likelihood function, it has a crucial property of being permutation invariant. 
In \Cref{sec:main} we show that $\hat{\tau}_n$ accurately estimates $\tau^*$ achieving a minimax optimal rate even in the heterogeneous noise setting. 

\paragraph{Link to standardization.} 
Let us now demonstrate that the estimators $\hat{\tau}_n$ and $\hat{\beta}_n$ implicitly apply standardization to the observations $\bX$ and $\bXdiese$. 
To this end, let $\tilde{\bX} = \{\tilde{X}_i\}_{i=1}^n$ and $\bXtildediese = \{\Xtildediese_i\}_{i=1}^n$ denote the standardized versions of $\bX$ and $\bXdiese$, respectively. For each $i \in [n]$ define
\begin{align}
    &\tilde{X}_i := \frac{X_i - \Xbar_n}{\sqrt{\sum_{i=1}^n \| X_i - \Xbar_n\|_2^2}},\\
    &\Xtildediese_i := \frac{\Xdiese_i - \Xbardiese_n}{\sqrt{\sum_{i=1}^n \| \Xdiese_i - \Xbardiese_n\|_2^2}}.
\end{align}
Notice that the LSL estimator for $\pi^*$ for the standardized observations $\tilde{\bX}$ and $\bXtildediese$ is defined as follows:
\begin{align}
    \hat{\pi}_n^{\textup{LSL}} &\in \argmin_{\pi \in \Sn} \sum_{i=1}^n \log \big(\|\tilde{X}_i -   \Xtildediese_{\pi(i)} \|_2^2\big).
\end{align}
Since $\sqrt{\sum_{i=1}^n \| X_i - \Xbar_n\|_2^2}$ is permutation invariant, the latter optimization problem is equivalent to \eqref{eq:pi-tau-mle}.

%% file: uai-sections/experiments.tex

\section{Numerical Experiments} \label{sec:computational}

In this section, we present numerical experiments on synthetic data to validate the main theoretical findings of our work. \Cref{fig:matching_viz} visualizes two $100$-dimensional synthetic datasets projected onto a shared 2D space using PCA. The underlying model for these datasets follows \eqref{eq:model} with the affine conditions specified in \eqref{eq:affine-assumptions}.
The results demonstrate that the baseline Least Sum of Logarithms (LSL) method, which does not account for the scalar affine transformation, fails to recover the correct matching. In contrast, our proposed method  estimates the unknown permutation with high accuracy.

The remainder of this section is organized as follows. 
In \Cref{sec:experiment-taustar}, we discuss the estimation procedure for $\taustar$ using $\tauhat_n$. 
Subsequently, \Cref{subsec:exp_permutation_recovery} presents numerical results that illustrate the findings of \Cref{thm:pi_recovery}, with a particular focus on the error introduced by the affine transformation - the component not present in prior work.

\subsection{Estimation Error of \texorpdfstring{$\tau^*$}{Tau}}\label{sec:experiment-taustar}
We demonstrate the theoretical result in \Cref{thm:tau-affine} through synthetic experiments conducted under the affine model \eqref{eq:model}--\eqref{eq:affine-assumptions}. The data is generated with a scale parameter $\tau^* = 3$ and a shift $\betastar = 0$. The distortion-to-noise ratio, $\lambda = \|\bmu\|_2/\|\bsigma\|_2$, is controlled by adjusting the spread of the parameters $\btheta$, and the noise magnitudes $\bsigma$ are specified separately in each regime below. Beyond the resulting value of $\lambda$, the specific configuration of the sets of vectors $\btheta$ and $\btdiese$ does not affect the estimation of $\taustar$.
The scale parameter $\tau^*$ is estimated using the estimator $\hat{\tau}_n$, as defined in \eqref{eq:intro-tau-hat}.
We conduct two experimental regimes, each averaged over 1000 trials.

In the first regime, shown in the top plot of \Cref{fig:error_tau_joint}, we vary the relative noise magnitude $\alpha_\sigma = \sigma_{\max}/\|\bsigma\|_2$ while keeping $(n, d, \lambda) = (2000, 20, 1)$ fixed. We do so by concentrating the total noise budget $\|\bsigma\|_2$ on a varying number $k$ of equally noisy coordinates while keeping the remaining magnitudes negligible, which sets $\alpha_\sigma \asymp 1/\sqrt{k}$ and sweeps it across its range $[n^{-1/2}, 1]$; we work in the regime $n \gg d$, so that the bias term $2d/(n(\lambda^2+d))$ is negligible. The empirical error then grows linearly in $\alpha_\sigma$, saturating the leading term $\alpha_\sigma/\sqrt{\lambda^2 + d}$ of \Cref{thm:tau-affine}.
In the second regime, depicted in the bottom plot of \Cref{fig:error_tau_joint}, we simultaneously increase both $d$ and $\lambda$ to explore a wide range of values for the joint quantity $\lambda^2 + d$, while fixing the sample size at $n = 100$ and using homogeneous noise (so that $\alpha_\sigma = n^{-1/2}$ is held fixed).
Notably, the empirical error closely follows the theoretical reference slope of $1/\sqrt{\lambda^2 + d}$, when plotted against $\lambda^2 + d$.
This confirms the sharpness of the bound of \Cref{thm:tau-affine}, demonstrating that the empirical error is indeed governed by the two quantities $\alpha_\sigma$ and $\lambda^2 + d$ through its leading term $\alpha_\sigma/\sqrt{\lambda^2 + d}$.

\begin{figure}[t]
    \centering
    \includegraphics[width=0.48\textwidth]{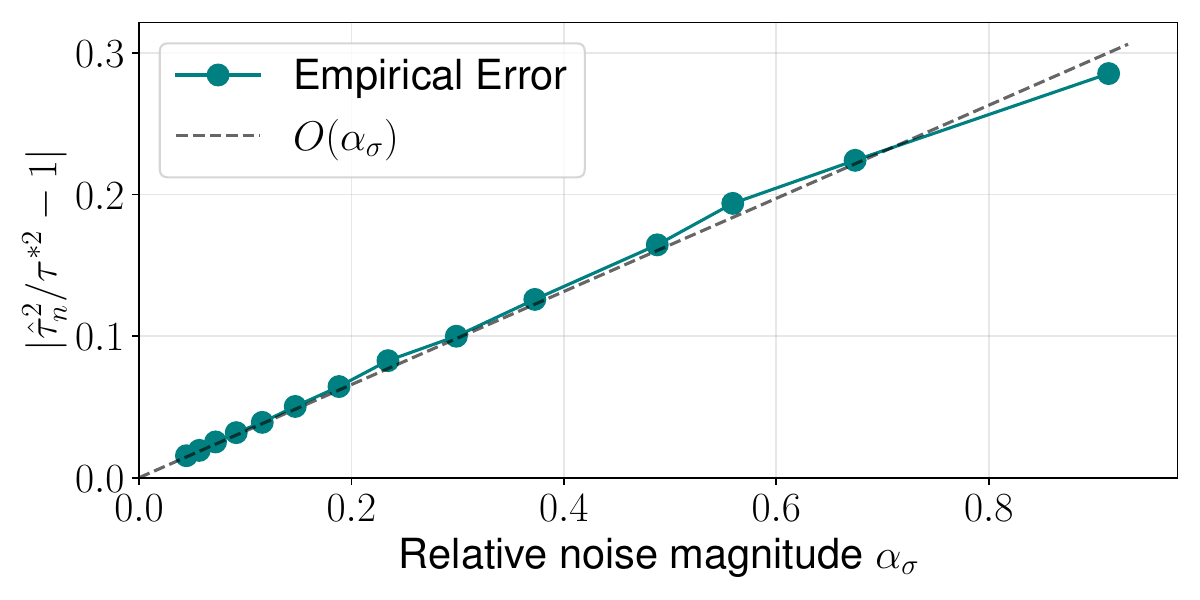}
    \includegraphics[width=0.48\textwidth]{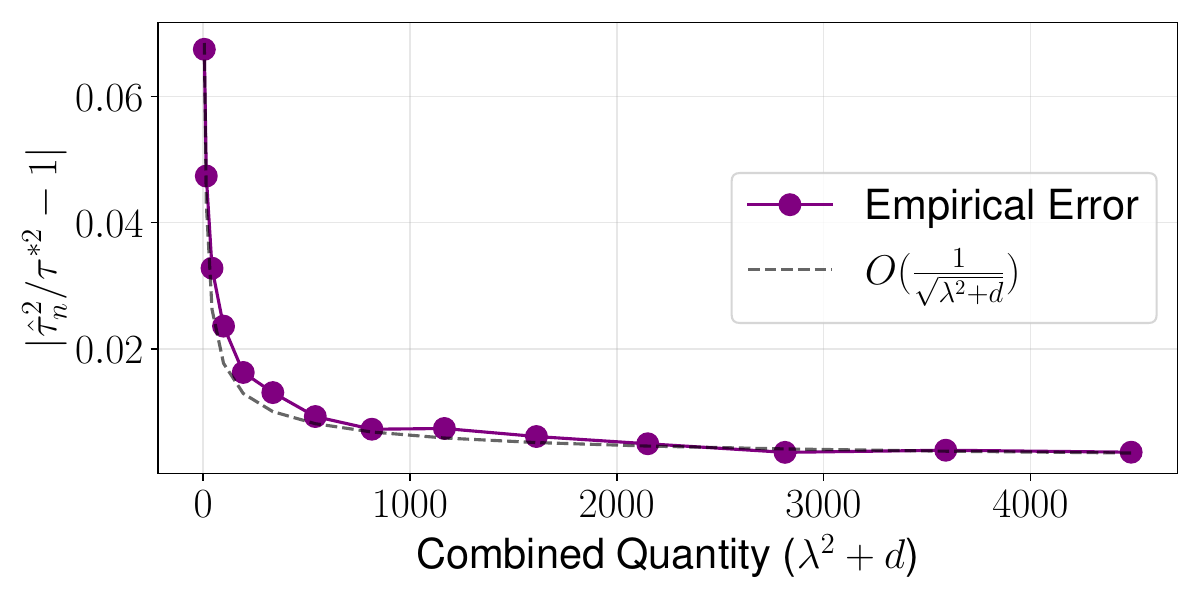}
    \caption{Estimation error versus relative noise magnitude $\alpha_\sigma$ \textit{(Top)} and joint quantity $\lambda^2 + d$ \textit{(Bottom)}. The empirical error closely follows the theoretical linear scaling in $\alpha_\sigma$ and the slope $1/\sqrt{\lambda^2 + d}$, validating the sharpness of the bounds in \Cref{thm:tau-affine}.}
    \label{fig:error_tau_joint}
\end{figure}

\subsection{Permutation Recovery and Necessity of the \texorpdfstring{$\sqrt{\rhosigma}$}{Square Root of Rho Sigma} Term}
\label{subsec:exp_permutation_recovery}

\begin{figure}[t]
    \centering
    \includegraphics[width=0.48\textwidth]{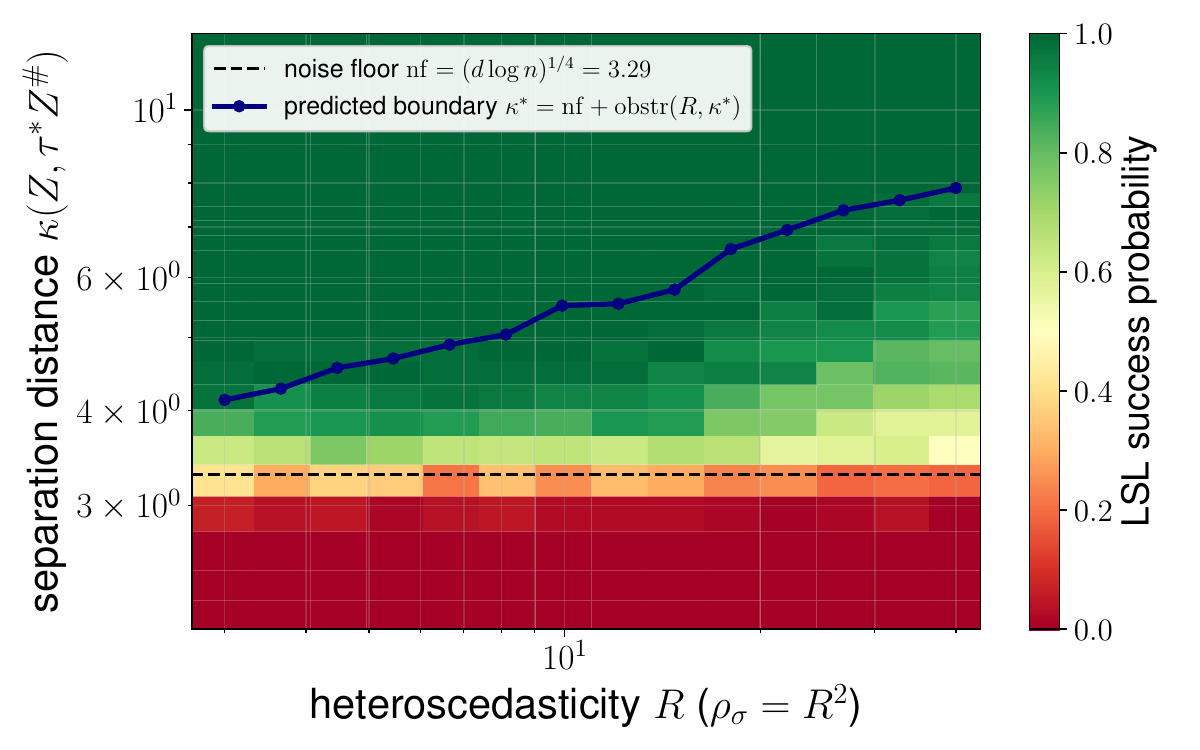}
    \caption{Empirical $\Pr(\hat{\pi}_n^{\textup{LSL}} = \pi^*)$ on the adversarial
    construction, over heteroscedasticity $R$ (with $\rhosigma = R^2$) and the separation
    distance $\kappaZ$; here $n = 2500$, $d = 15$, and each cell is averaged over $80$ trials.
    The dashed line marks the noise floor $\mathrm{nf} = (d\log n)^{1/4}$, the separation that
    already suffices when $\tau^*$ is known. The solid curve is the predicted threshold
    $\kappa^\star = \mathrm{nf} + \mathrm{obstr}(R,\kappa^\star)$ with
    $\mathrm{obstr} = \etau\max_{i}\|\mu_i\|_2/s_i$. Recovery fails in a band that rises above
    the noise floor as $\rhosigma$ grows, tracking the third term of \Cref{thm:pi_recovery}.}
    \label{fig:lsl_lower_bound}
\end{figure}

\Cref{thm:pi_recovery} guarantees perfect recovery once $\kappaZ$ exceeds the sum of three
contributions. The {noise floor} $(d\log n)^{1/4}\vee\sqrt{\log n}$, inherent to the
Gaussian and $\chi^2$ fluctuations, are standard and consistent with prior work \citep{collier2016minimax}. 
The third term $\sim \sqrt{\rhosigma}$ is due to the estimation of $\tau^*$ and comes downstream as an upper bound of the term $\etau\max_{i}\|\mu_i\|_2/s_i$, where $\etau$ denotes the relative error of the scale estimate $\tauhat_n$. 
Below we focus on this term, which is specific to the unknown affine transformation. We ask whether it is an artifact of our proof or an intrinsic limitation of the standardization-based LSL procedure, and we answer the latter with an explicit construction.

\paragraph{Construction.}
Fix a unit vector $V \in \mathbb{S}^{d-1}$ and set $\tau^* = 1$, $\beta^* = 0$, and
$\pi^* = \id$. We place an \emph{adversarial pair} $\theta_1 = MV$, $\theta_2 = (M+\Delta)V$ with
$\sigma_1 = \sigma_2 = 1$, a mirror pair $\theta_3 = -MV$, $\theta_4 = -(M+\Delta)V$ that keeps
$\bar{\theta}_n = 0$, and $n-4$ low-signal \emph{inflator} pairs $\theta_{4+2k-1} = r u_k$,
$\theta_{4+2k} = -r u_k$ along directions $u_k \in V^\perp$ with large noise $\sigma_i = R$.
The inflators carry almost all of the noise mass, so that $\rhosigma = R^2$ and $\tauhat_n$ is
least accurate precisely along the adversarial direction $V$. We set $\Delta = \sqrt{2}\,s_1\kappaZ$ so that the adversarial pair attains the minimum separation $\kappaZ$, and
$M = C R\sqrt{nd}$ so that $\max_{i}\|\mu_i\|_2/s_i$ saturates the bound
\eqref{eq:thm3-rate}, making the term $\etau\max_{i}\|\mu_i\|_2/s_i \asymp
\sqrt{\rhosigma}$ the important quantity.

\paragraph{Results.}
\Cref{fig:lsl_lower_bound} reports the empirical recovery probability of $\hat{\pi}_n^{\textup{LSL}}$ over the $(R, \kappaZ)$ plane. Were the $\sim\sqrt{\rhosigma}$ term an artifact, recovery would succeed everywhere above the dashed line; instead a wide band of failure persists above it, and its upper edge climbs steadily with $\rhosigma$, closely tracked by $\etau\max_{i}\|\mu_i\|_2/s_i$. 
\Cref{fig:lsl_lb_slices} exposes the mechanism behind this band. Its top panel plots the obstruction
$\mathrm{obstr}(\kappaZ, R) = \etau\max_{i}\|\mu_i\|_2/s_i$ against the threshold line, and its bottom panel shows the matching success curves on the same $\kappaZ$ axis. For each $R$, success rises to one precisely once the obstruction
curve drops below the threshold line; a larger $R$ carries a larger obstruction and therefore fails
over a wider range of separations, so the highest-$R$ curve transitions last. Since this
obstruction is $\asymp \sqrt{\rhosigma}$ in our family, the failure boundary scales as
$(d\log n)^{1/4}\vee\sqrt{\rhosigma}$: the third term of \Cref{thm:pi_recovery}, cannot be removed by a sharper analysis of LSL. 

See more experimental results on real-world translation data from the OPUS-100 dataset \citep{opus} in \Cref{subsec:real_lang} of the supplementary material of this work.

\begin{figure}[t]
    \centering
    \includegraphics[width=0.48\textwidth]{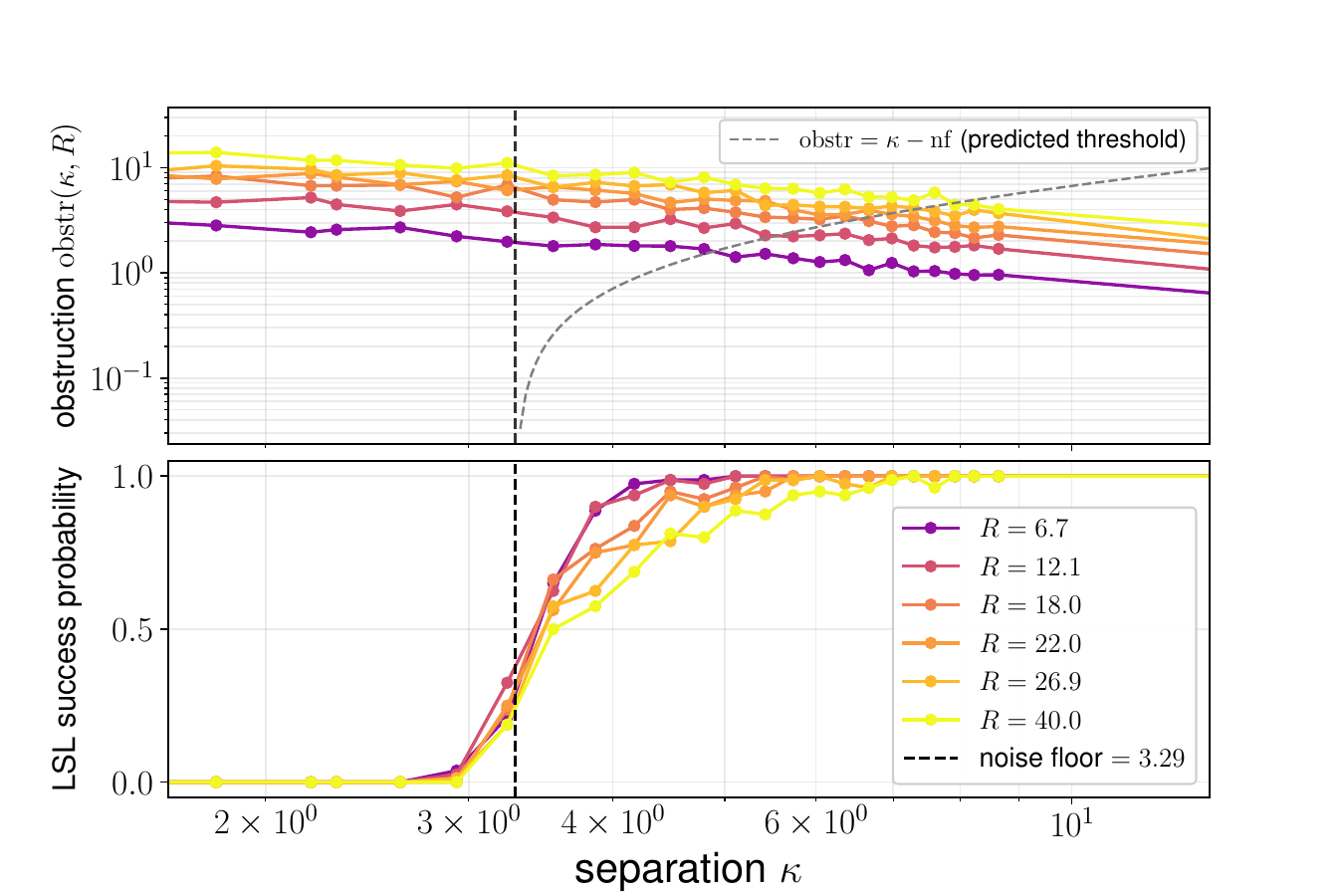}
    \caption{\textit{(Top)} Empirical obstruction $\mathrm{obstr}(\kappaZ, R) =
    \etau\max_{i}\|\mu_i\|_2/s_i$ for several values of $R$, against the predicted threshold
    line (dashed). \textit{(Bottom)} Corresponding LSL
    success curves on the same $\kappaZ$ axis. For each $R$, success rises to one just past the
    $\kappaZ$ at which its obstruction curve falls below the threshold line; larger $R$ (hence
    larger $\rhosigma$) crosses later, widening the failure region. The vertical dashed line is
    the noise floor $\mathrm{nf} \approx (d\log n)^{1/4}$.}
    \label{fig:lsl_lb_slices}
\end{figure}

%% file: uai-sections/Conclusion.tex
\section{Conclusion}\label{sec:conclusion}

We have showed that the matching problem under mild heteroscedasticity and a structural assumption that the true feature vectors are connected by an unknown scalar affine transformation is no harder, in a minimax sense, than that of identity transformation.
In other words, we showed that it is indeed possible to learn a transformation of the form $\tau I_d + \beta$ between two noisy sequences $\bX$ and $\bXdiese$ sufficiently well in order to perfectly recover the underlying matching $\pi^*$.
Moreover, the minimax rate of the minimal separation distance remains the same, matching, up to multiplicative constants, the lower bound in the non-affine setting.
However, in the general heterogeneous setting, where the noise magnitudes are allowed to grow with sample size $n$, we showed that the rate changes and scales linearly with the square root of the maximal ratio of noise variances.
This phenomenon is rooted in the estimation error of the scaling factor which depends on the maximal ratio of noise variances. Therefore, a larger minimal separation distance is required to be able to perfectly recover the unknown permutation.
Interestingly, learning an unknown scalar affine transformation is equivalent to standardizing the sequences $\bX$ and $\bXdiese$, that is, subtracting the sample mean and dividing by the square root of the total sample variance.

%% file: uai-sections/Appendix.tex
\appendix

\section{Postponed proofs from Section \ref{sec:main}}\label{app:A}
We first state a few standard lemmas that are used to control the tails of Gaussian and $\chi^2$ distributions. 
\begin{lemma}\label{lem:massart}{\citep[Lemma 4.1]{LaurentMassart2000}}
        Let $(Y_1, \ldots, Y_n)$ be i.i.d. Gaussian random variables with mean $0$ and variance $1$. Let $\ba = \brr{a_1, \ldots, a_n}$ be a collection of nonnegative constants. Let $T$ be the weighted sum, defined as 

    \begin{align}
        T = \sum_{i=1}^n a_i(Y_i^2 - 1).
    \end{align}
    Then, the following inequalities hold for any positive $x$
    \begin{align}
        \mathbb{P}(T \geq 2\norm{a}_2 \sqrt{x} + 2\norm{a}_{\infty} x) \leq \exp(-x), \qquad
        \mathbb{P}(T \leq -2\norm{a}_2 \sqrt{x}) \leq \exp(-x).
    \end{align}
\end{lemma}

\begin{lemma}[Gaussian concentration inequality]
    Let $(Y_1, \ldots, Y_n)$ be independent Gaussian random vectors in $\RR^d$ with mean $0$ and covariance matrices $\sigma_i^2 I_d$. Let $\ba = (a_1, \dots, a_n)$ be a collection of fixed arbitrary vectors in $\RR^d$. Let $T$ be the weighted sum, defined as 
    \begin{align}
        T = \sum_{i=1}^n a_i^\top Y_i.
    \end{align}
    Then, for any positive $x$ we have 
    \begin{align}
        \Prob(|T| \ge x) \le 2\exp(-x^2/(2\sum_{i=1}^n \sigma_i^2 \|a_i\|^2))
    \end{align}
\end{lemma}

\subsection{Proof of Proposition \ref{prop:mean-diff}}\label{sub:proof_of_prop_mean_diff}

\begin{proof}
Recall that $Z_i$ are Gaussian vectors with parameters $(\mu_i, s_i^2 I_d)$, and decompose them into $Z_i = \mu_i + \Xibar_{i,n}$, where
\begin{align}
    \Xibar_{i,n} := \sigma_i \xi_i - \frac{1}{n}\sum_{k=1}^n \sigma_k\xi_k.
\end{align}
Then the sum of squared $\ell_2$ norms of $Z_1, \ldots, Z_n$ can be decomposed as
\begin{align}
    \sum_{i=1}^n \|Z_i\|_2^2
    = \| \bmu \|_2^2 + 2\sum_{i=1}^n \sigma_i\mu_i^\top \xi_i + \sum_{i=1}^n  \| \Xibar_{i,n}\|_2^2,
\end{align}
where we used that summing $\mu_i^\top \Xibar_{i,n}$ over $[n]$ is the same as summing $\sigma_i \mu_i^\top \xi_i$ over the same set, since $\sum_{i=1}^n \mu_i = 0$. We bound the last two terms separately.

The term $\sum_{i=1}^n \sigma_i\mu_i^\top \xi_i$ has a Gaussian distribution with mean zero and variance $\sum_{i=1}^n \sigma_i^2 \| \mu_i\|_2^2 \le \sigma_{\max}^2\|\bmu\|_2^2$. Thus, by the Gaussian concentration inequality, with probability at least $1-\delta/2$,
\begin{align}
    \sum_{i=1}^n \sigma_i \mu_i^\top \xi_i \le \sigma_{\max}\|\bmu\|_2\sqrt{2\log(2/\delta)}.
\end{align}
As for the term $\sum_{i=1}^n \|\Xibar_{i, n}\|_2^2$, we write it as a quadratic form in the underlying noise. Stack the noise vectors into $\bxi := (\xi_1^\top, \ldots, \xi_n^\top)^\top \sim \mathcal{N}(0, I_{nd})$, and let $D := \mathrm{diag}(\sigma_1, \ldots, \sigma_n)$ scale the $n$ blocks while $P := I_n - \frac{1}{n}\mathbf{1}\mathbf{1}^\top$ is the centering projection that subtracts the block average. The two operations act blockwise through the Kronecker product: $(D\otimes I_d)\bxi$ stacks the scaled vectors $\sigma_i\xi_i$, and applying $P\otimes I_d$ subtracts their average, so that
\begin{align}
    \big[(PD\otimes I_d)\bxi\big]_i = \sigma_i\xi_i - \frac{1}{n}\sum_{k=1}^n \sigma_k\xi_k = \Xibar_{i,n}.
\end{align}
Therefore $\sum_{i=1}^n \|\Xibar_{i, n}\|_2^2 = \big\|(PD\otimes I_d)\bxi\big\|_2^2 = \bxi^\top (PD\otimes I_d)^\top (PD\otimes I_d)\,\bxi$. By the Kronecker identities $(M\otimes I_d)^\top = M^\top\otimes I_d$ and $(M_1\otimes I_d)(M_2\otimes I_d) = (M_1M_2)\otimes I_d$, together with the symmetry of $D$ and $P$ and the idempotence $P^2 = P$, the inner matrix equals $(PD)^\top(PD)\otimes I_d = (DPD)\otimes I_d$. Moreover, $DPD = D^\top P D$ is a congruence of the projection $P\succeq 0$ and is therefore positive semi-definite. Writing $A := (DPD)\otimes I_d$, we conclude that
\begin{align}
    \sum_{i=1}^n \|\Xibar_{i, n}\|_2^2 = \bxi^\top A \bxi, \qquad A \succeq 0.
\end{align}
Using $\mathrm{tr}(M\otimes I_d) = d\,\mathrm{tr}(M)$, $\|M\otimes I_d\|_F = \sqrt{d}\,\|M\|_F$ and $\|M\otimes I_d\|_{\mathrm{op}} = \|M\|_{\mathrm{op}}$, a direct computation yields
\begin{align}
    \mathrm{tr}(A) &= d\,\mathrm{tr}(PD^2) = d\,\frac{n-1}{n}\|\bsigma\|_2^2,\\
    \|A\|_F^2 &= d\,\mathrm{tr}\big((DPD)^2\big) = d\Big(\frac{n-2}{n}\|\bsigma\|_4^4 + \frac{1}{n^2}\|\bsigma\|_2^4\Big) \le d\,\|\bsigma\|_4^4,\\
    \|A\|_{\mathrm{op}} &= \|DPD\|_{\mathrm{op}} \le \|D\|_{\mathrm{op}}^2\|P\|_{\mathrm{op}} = \sigma_{\max}^2,
\end{align}
where the bound on $\|A\|_F^2$ uses $\|\bsigma\|_2^4 \le n\|\bsigma\|_4^4$. We now reduce the quadratic form $\bxi^\top A \bxi$ to a weighted sum of squared Gaussians. Since $A$ is symmetric and positive semi-definite, the spectral theorem gives $A = U\Lambda U^\top$, where $U$ is orthogonal and $\Lambda = \mathrm{diag}(\lambda_1, \ldots, \lambda_{nd})$ collects the eigenvalues $\lambda_k \ge 0$. Let $\eta := U^\top \bxi$. Since $\bxi \sim \mathcal{N}(0, I_{nd})$ and $U$ is orthogonal, $\eta$ is again centered Gaussian with covariance $U^\top U = I_{nd}$, so $\eta \sim \mathcal{N}(0, I_{nd})$ and its coordinates $\eta_1, \ldots, \eta_{nd}$ are independent $\mathcal{N}(0,1)$ variables. In this eigenbasis,
\begin{align}
    \bxi^\top A \bxi = \eta^\top \Lambda \eta = \sum_{k=1}^{nd} \lambda_k \eta_k^2
    \qquad \text{and} \qquad
    \mathrm{tr}(A) = \sum_{k=1}^{nd} \lambda_k,
\end{align}
so that
\begin{align}
    \bxi^\top A \bxi - \mathrm{tr}(A) = \sum_{k=1}^{nd} \lambda_k(\eta_k^2 - 1).
\end{align}
The right-hand side is a weighted sum of independent centered chi-squared variables with nonnegative weights $a_k = \lambda_k$, so \Cref{lem:massart} applies. Its norms are exactly the matrix norms computed above, namely $\sum_{k} a_k^2 = \sum_k \lambda_k^2 = \|A\|_F^2$ and $\max_k a_k = \max_k \lambda_k = \|A\|_{\mathrm{op}}$. Applying the upper bound of \Cref{lem:massart} with $x = \log(2/\delta)$, we have with probability at least $1-\delta/2$,
\begin{align}
    \sum_{i=1}^n \|\Xibar_{i, n}\|_2^2
    \le d\,\frac{n-1}{n}\|\bsigma\|_2^2 + 2\|\bsigma\|_4^2\sqrt{d\log(2/\delta)} + 2\sigma_{\max}^2\log(2/\delta).
\end{align}
Combining these two bounds via the union bound, and using $d\,\frac{n-1}{n}\|\bsigma\|_2^2 \le d\|\bsigma\|_2^2$, we arrive at the following upper bound, which holds with probability at least $1-\delta$:
\begin{align}
    \sum_{i=1}^n \|Z_i\|_2^2
    &\le \|\bmu\|_2^2 + d\|\bsigma\|_2^2
    + 2\sigma_{\max}\|\bmu\|_2\sqrt{2\log(2/\delta)}
    + 2\|\bsigma\|_4^2\sqrt{d\log(2/\delta)}
    + 2\sigma_{\max}^2 \log(2/\delta),
\end{align}
completing the first part of the proof.

To prove the lower bound, we apply the lower-bound parts of the Gaussian concentration inequality and \Cref{lem:massart}. With probability at least $1-\delta/2$,
\begin{align}
    \sum_{i=1}^n \sigma_i \mu_i^\top \xi_i \ge -\sigma_{\max}\|\bmu\|_2\sqrt{2\log(2/\delta)},
\end{align}
and also with probability at least $1-\delta/2$,
\begin{align}
    \sum_{i=1}^n \|\Xibar_{i, n}\|_2^2
    \ge d\,\frac{n-1}{n}\|\bsigma\|_2^2 - 2\|\bsigma\|_4^2\sqrt{d\log(2/\delta)}.
\end{align}
Combining these via the union bound and using $d\,\frac{n-1}{n}\|\bsigma\|_2^2 = d\|\bsigma\|_2^2 - \frac{d}{n}\|\bsigma\|_2^2$, we obtain that with probability at least $1-\delta$,
\begin{align}
    \sum_{i=1}^n \|Z_i\|_2^2
    &\ge \|\bmu\|_2^2 + d\|\bsigma\|_2^2
    - 2\sigma_{\max}\|\bmu\|_2\sqrt{2\log(2/\delta)}
    - 2\|\bsigma\|_4^2\sqrt{d\log(2/\delta)}
    - \frac{d}{n}\|\bsigma\|_2^2,
\end{align}
which completes the proof.
\end{proof}

\subsection{Proof of Theorem \ref{thm:tau-affine}}\label{sub:proof_of_thm_tau_affine}

Recall the definition of $\hat{\tau}_n$ from \eqref{eq:intro-tau-hat}:
\begin{align}
    \hat{\tau}_n 
    :=  \brr{\frac{\sum_{i=1}^n \| X_i - \bar{X}_n\|_2^2}{\sum_{i=1}^n \| \Xdiese_i - \Xbardiese_n\|_2^2}}^{1/2} = \brr{\frac{\sum_{i=1}^n \| Z_i\|_2^2}{\sum_{i=1}^n \| \Zdiese_i\|_2^2}}^{1/2}.
\end{align}
Then,
\begin{align}
\frac{\hat{\tau}_n^2 }{\taustar^2}- 1  = \frac{\sum_{i=1}^n (\| Z_i\|_2^2 - \| \Zdiese_i\|_2^2)}{\taustar^2\sum_{i=1}^n \| \Zdiese_i\|_2^2} .
\end{align}

The proof consists of simultaneously upper and lower bounding the numerator and the denominator, then combining these bounds using the union bound. 
In what follows we apply \Cref{prop:mean-diff} four times: upper and lower bounds for both centralized sequences $\bZ$ and $\bZdiese$ in each of the cases $\tauhat_n > \taustar$ and $\tauhat_n < \taustar$, yielding an upper bound for the absolute value of the difference between $\tauhat_n^2$ and $\taustar{}^2$. Notice that, for all $i \in [n]$ we have that $Z_i \perp \Zdiese_{\pi^*(i)}$, and $Z_i, \taustar \Zdiese_{\pi^*(i)} \sim \mathcal{N}(\mu_i, s_i^2 I_d)$. In what follows, without loss of generality, we assume $\taustar = 1$ for notational convenience.
Using the lower bound of \Cref{prop:mean-diff} for $\bZdiese$ we get that with probability at least $1-\delta$, 

\begin{align}\label{eq:LB-Zdiese}
    \sum_{i=1}^n \|\Zdiese_i\|_2^2
    &\ge \|\bmu\|_2^2 + d\|\bsigma\|_2^2 - 2\sigma_{\max}\|\bmu\|_2\sqrt{2\log(2/\delta)}
    - 2\|\bsigma\|_4^2\sqrt{d\log(2/\delta)} - \frac{d}{n}\|\bsigma\|_2^2.
    \end{align}
Similarly, for the upper bound we have, with probability at least $1-\delta$,

\begin{align}\label{eq:UB-Z}
    \sum_{i=1}^n \|Z_i\|_2^2
    &\le \|\bmu\|_2^2 + d\|\bsigma\|_2^2 + 2\sigma_{\max}\|\bmu\|_2\sqrt{2\log(2/\delta)} 
     + 2\|\bsigma\|_4^2\sqrt{d\log(2/\delta)} + 2\sigma_{\max}^2 \log(2/\delta),
    \end{align}
Both \eqref{eq:LB-Zdiese} and \eqref{eq:UB-Z} follow from \Cref{prop:mean-diff} applied at confidence level $\delta$.

Combining these bounds we get that
\begin{align}
    \sum_{i=1}^n \big(\|Z_i\|_2^2 - \| \Zdiese_i\|_2^2\big)  
    &\le 4\sigma_{\max}\|\bmu\|_2\sqrt{2\log(2/\delta)}
    + 4\|\bsigma\|_4^2\sqrt{d\log(2/\delta)}
    + 2\sigma_{\max}^2 \log(2/\delta) 
    + \frac{d}{n}\|\bsigma\|_2^2
\end{align}
with probability at least $1-2\delta$, holding simultaneously with \eqref{eq:LB-Zdiese}.
Thus, with probability at least $1-4\delta$,
\begin{align}
    \Big|\sum_{i=1}^n \big(\|Z_i\|_2^2 - \| \Zdiese_i\|_2^2\big)\Big|
    &\le 4\sigma_{\max}\|\bmu\|_2\sqrt{2\log(2/\delta)}
    + 4\|\bsigma\|_4^2\sqrt{d\log(2/\delta)}
    + 2\sigma_{\max}^2 \log(2/\delta) 
    + \frac{d}{n}\|\bsigma\|_2^2,
\end{align}
which holds simultaneously with \eqref{eq:LB-Zdiese}.
Recall that $\alpha_\sigma = \sigma_{\max}/\|\bsigma\|_2$ and $\lambda = \|\bmu\|_2/\|\bsigma\|_2$.
Dividing the previous display by $\|\bsigma\|_2^2$ and using the following identities
\begin{align}
    \frac{\sigma_{\max}\|\bmu\|_2}{\|\bsigma\|_2^2} = \alpha_\sigma\lambda,\quad
    \frac{\sigma_{\max}^2}{\|\bsigma\|_2^2} = \alpha_\sigma^2,\quad
    \frac{\|\bsigma\|_4^2}{\|\bsigma\|_2^2} \le \alpha_\sigma,
\end{align}
where the last follows from $\|\bsigma\|_4^4 \le \sigma_{\max}^2\|\bsigma\|_2^2$. Combining the linear cross-terms via $\sqrt{2}\lambda + \sqrt{d} \le \sqrt{2}(\lambda+\sqrt{d})$, we obtain
\begin{align}\label{eq:num-alpha}
    \Big|\sum_{i=1}^n \big(\|Z_i\|_2^2 - \|\Zdiese_i\|_2^2\big)\Big|
    &\le \|\bsigma\|_2^2 \bigg(
    4\alpha_\sigma(\lambda+\sqrt{d})\sqrt{2\log(2/\delta)}
    + 2\alpha_\sigma^2\log(2/\delta)
    + \frac{d}{n}
    \bigg).
\end{align}
Applying the same identities to the right-hand side of \eqref{eq:LB-Zdiese}, we get that with probability at least $1-\delta$
\begin{align}\label{eq:denom-alpha}
    \sum_{i=1}^n \|\Zdiese_i\|_2^2
    &\ge \|\bsigma\|_2^2\bigg(
    \lambda^2 + d
    - 2\alpha_\sigma(\lambda+\sqrt{d})\sqrt{2\log(2/\delta)}
    - \frac{d}{n}
    \bigg).
\end{align}

The conditions of the theorem ensures that the obtained lower bound \eqref{eq:denom-alpha} is indeed larger than the upper bound from \eqref{eq:num-alpha} for any $\lambda \ge 0$.
Notice that otherwise the consistent estimation of $\taustar$ is impossible. 
In order for $|\hat{\tau}_n^2/\taustar^2 - 1|$ to be smaller than one, the right-hand side of \eqref{eq:denom-alpha} must dominate that of \eqref{eq:num-alpha}, which is equivalent to
\begin{align}\label{eq:condition-n-delta-alpha}
    h(\lambda) := \lambda^2 + d
    - 6\alpha_\sigma(\lambda+\sqrt{d})\sqrt{2\log(2/\delta)}
    - 2\alpha_\sigma^2\log(2/\delta)
    - \frac{2d}{n} \ge 0,
\end{align}
for every $\lambda \ge 0$. 
The minimum of this quadratic in $\lambda$, attained at $\lambda^\star = 3\alpha_\sigma\sqrt{2\log(2/\delta)} \ge 0$, gives
\begin{align}
    h(\lambda^\star)
    = d - 20\alpha_\sigma^2\log(2/\delta)
    - 6\alpha_\sigma\sqrt{2d\log(2/\delta)}
    - \frac{2d}{n}.
\end{align}
By the AM-GM inequality, $6\alpha_\sigma\sqrt{2d\log(2/\delta)} = \sqrt{d\cdot 72\alpha_\sigma^2\log(2/\delta)} \le d/2 + 36\alpha_\sigma^2\log(2/\delta)$, hence
\begin{align}
    h(\lambda^\star)
    \ge \frac{d}{2} - 56\alpha_\sigma^2\log(2/\delta) - \frac{2d}{n}.
\end{align}
Now given that $n \ge 8$ and $\log(2/\delta) \le d/(224\alpha_\sigma^2)$, one can check that
\begin{align}
    56\alpha_\sigma^2\log(2/\delta) \le \frac{d}{4},\quad
    \frac{2d}{n} \le \frac{d}{4},
\end{align}
and therefore $h(\lambda^\star) \ge 0$.
Since for any positive $a$ and $b$ such that $a \le b$ and $x > 0$ we have $\frac{a}{b} \le \frac{a+x}{b+x}$, we arrive at the following bound for the estimation error of $\taustar$, with probability $1-4\delta$
\begin{align}\label{eq:thm1-tau-error-alpha}
    \Big|\tauhat_n^2 - 1\Big|
    &= \frac{\Big|\sum_{i=1}^n \|Z_i\|_2^2 - \sum_{i=1}^n \|\Zdiese_i\|_2^2\Big|}
           {\sum_{i=1}^n \|\Zdiese_i\|_2^2} \\
    &\le \frac{
        4\alpha_\sigma(\lambda+\sqrt{d})\sqrt{2\log(2/\delta)}
        + 2\alpha_\sigma^2\log(2/\delta)
        + n^{-1}d
    }{
        \lambda^2 + d - 2\alpha_\sigma(\lambda+\sqrt{d})\sqrt{2\log(2/\delta)} - n^{-1}d
    } \\
    &\le \frac{
        6\alpha_\sigma(\lambda+\sqrt{d})\sqrt{2\log(2/\delta)}
        + 2\alpha_\sigma^2\log(2/\delta)
        + 2n^{-1}d
    }{\lambda^2 + d} \\
    &\le 12\alpha_\sigma\sqrt{\frac{\log(2/\delta)}{\lambda^2 + d}}
    + \frac{2\alpha_\sigma^2\log(2/\delta) + 2n^{-1}d}{\lambda^2 + d},
\end{align}
where we used $\lambda + \sqrt{d} \le \sqrt{2(\lambda^2 + d)}$ in the final step, concluding the proof.

\subsection{Proof of Theorem \ref{thm:pi_recovery}}\label{sub:proof_of_thm_pi_recovery}

We first state two key lemmas that are used to control the tails of standard Gaussian and normalized $\chi^2$ distributions. Their proofs are postponed to Sections \ref{sec:proof-lem-zeta-1} and \ref{sec:proof-lem-zeta-2}, respectively. 

\begin{lemma}\label{lem:zeta1_bound}
Let $\eta_1, \ldots, \eta_n$ and $\eta^\prime_1, \ldots, \eta^\prime_n$   be two mutually independent sets of $d$-dimensional Gaussian random vectors, with $\eta_i, \eta^\prime_i \sim N(0, \omega_i^2 I_d)$. 
Notice that $\eta_i$ and $\eta_j$ may be correlated. 
Fix a set of unit vectors $u_{i,j} \in \mathbb{S}^{d-1}$, with $i,j\in[n]$ and $\varepsilon \in (0, 1/2)$. 
Define
\begin{equation}
\Phi_1 := \sup_{a \in [1-\varepsilon, 1+\varepsilon]} \max_{i,j} \left|\frac{u_{i,j}^\top(\eta_i - a \eta^\prime_j)}{\sqrt{\omega_i^2 + a^2\omega_j^2}}\right|.
\end{equation}
Then for any $\delta \in (0,1)$, with probability at least $1-\delta$,
\[
\Phi_1 \leq \sqrt{2 \log\left(12 n^3 / \delta\right)} + \frac{2\varepsilon}{n} \sqrt{\log(8n/\delta)}.
\]
\end{lemma}
\begin{lemma}\label{lem:zeta2_bound}
Let $\eta_1, \ldots, \eta_n$ and $\eta^\prime_1, \ldots, \eta^\prime_n$ be two mutually independent sets of $d$-dimensional Gaussian random vectors, with $\eta_i, \eta^\prime_i \sim N(0, \omega_i^2 I_d)$. 
Notice that $\eta_i$ and $\eta_j$ may be correlated. For some $\varepsilon \in (0, 1/2)$, define
\begin{equation}
\Phi_2 := \sup_{a \in [1-\varepsilon, 1+\varepsilon]} \max_{i,j} \left|\frac{\|\eta_i - a\eta^\prime_j\|_2^2}{\omega_i^2 + a^2\omega_j^2} - d\right|.
\end{equation}
Then for any $\delta \in (0,1)$, with probability at least $1-\delta$,
\[
\Phi_2 \leq 2\sqrt{d \log(12 n^3 / \delta)} + 2\log(12 n^3 / \delta) + \frac{4\varepsilon}{n} \left( d + 2\sqrt{d \log(4n/\delta)} + 2\log(4n/\delta) \right).
\]
\end{lemma}

\begin{proof}[Proof of \Cref{thm:pi_recovery}]
    
Without loss of generality, we assume $\pi^*$ is the identity map $\id$ and $\tau^* = 1$. 
Recall that the observations $\bZ$ and $\bZdiese$ are the centralized versions of $\bX$ and $\bXdiese$, respectively. Recall as well the objective function of the LSL estimator from $\eqref{eq:pi-tau-mle}$:
\begin{equation}
    \ell_n(\pi,\bZ,\bZdiese) = \sum_{i=1}^n \log \Big(\|Z_i - \tauhat_n\Zdiese_{\pi(i)}\|_2^2\Big),
\end{equation}
up to some permutation independent constants.
We aim to bound the probability of the event $\Omega = \{ \hat{\pi} \neq \id\}$, where $\hat{\pi} = \hat\pi^{\rm LSL}_n= \argmin_{\pi\in\Sn}\ell_n(\pi,\bZ,\bZdiese) $. 
Our primary observation is that 
\begin{equation}
    \Omega \subset \cup_{\pi \neq \id} \Omega_{\pi}, \eqtext{where}
    \Omega_{\pi} = \brc{\pi \in \argmin_{\pi\in\Sn} \ell_n(\pi,\bZ, \bZdiese)}.
\end{equation}
Then, the following chain of inclusions holds:
\begin{align}
    \Omega  &\subset \bigcup_{\pi \neq \id} \Omega_{\pi}
    \subset \bigcup_{\pi \neq \id} \brc{\ell_n(\pi,\bZ, \bZdiese) \leq \ell_n(\id,\bZ, \bZdiese)}  \\
    &\subset \bigcup_{\pi \neq \id} \bigg\{\sum_{i : \pi(i)\not= i}\!\! \log \frac{\|Z_i - \tauhat_n \Zdiese_i \|_2^2}{\|Z_i - \tauhat_n \Zdiese_{\pi(i)}\|_2^2}\ge 0\bigg\}.
\end{align}
Jensen's inequality yields the following for every $t>0$:
\begin{align}
    \sum_{\pi(i)\neq i} \log\left(\frac{(1+t^2)s_i^2}{s_i^2+t^2 s_{\pi(i)}^2}\right)
    &= \sum_{i=1}^n \big(\log((1+t^2)s_i^2) - \log(s_i^2+t^2 s_{\pi(i)}^2) \big)\\
    &= \sum_{i=1}^n \left( \frac{\log((1+t^2)s_i^2)+t^2\log((1+t^2)s_{\pi(i)}^2)}{1+t^2}-\log(s_i^2+t^2 s_{\pi(i)}^2) \right) \le 0.
\end{align}
Setting $t = \tauhat_n$ yields 
\begin{align}
    \Omega &\subset \bigcup_{\pi \neq \id} \bigg\{\sum_{i : \pi(i)\not= i} \log \frac{\|Z_i - \tauhat_n \Zdiese_i \|_2^2}
     {\|Z_i - \tauhat_n \Zdiese_{\pi(i)}\|_2^2} - \sum_{i : \pi(i)\not= i}\log\bigg(\frac{(1+\tauhat_n^2)s_i^2}{s_i^2 + \tauhat_n^2 s_{\pi(i)}^2}\bigg)\ge 0\bigg\} \\
    &\subset \bigcup_{i=1}^n \bigcup_{j\not =i}\bigg\{\frac{\|Z_i - \tauhat_n \Zdiese_i \|_2^2}{(1+\tauhat_n^2)s_i^2}\ge\frac{\|Z_i - \tauhat_n \Zdiese_{j}\|_2^2}{s_i^2 + \tauhat_n^2 s_{j}^2}\bigg\}. 
\end{align}
Define the random event $\Omega_{\tau} := \{\tauhat_n\in \cE_\tau := [1-\etau,1+\etau]\}$, where $\etau$ is the upper bound on the estimation error of $\tauhat_n$ derived in \Cref{thm:tau-affine}. 
Then, intersecting with the event $\Omega_{\tau}$ we get the following inclusion:
\begin{align}\label{eq:inclusions}
    {\Omega} &\subseteq  \brc{\bigcup_{i=1}^n \bigcup_{j\not =i}\bigg\{\frac{\|Z_i - \tauhat_n \Zdiese_i \|_2^2}{(1+\tauhat_n^2)s_i^2}\ge\frac{\|Z_i - \tauhat_n \Zdiese_{j}\|_2^2}{s_i^2 + \tauhat_n^2 s_{j}^2}\bigg\}\cap \Omega_{\tau}}\cup \Omega_\tau^{\sf c}.
\end{align}
In order to upper bound the probability of the event in \eqref{eq:inclusions} we first upper and lower bound the left and right hand sides of the inequality from \eqref{eq:inclusions}, respectively. 
In other words, for the correct pair we present an upper bound on the left hand side, while for the incorrect pair  we compute bound it from below. 
From the definitions of $Z_i$ and $\Zdiese_j$, for any pair $(i, j) \in [n]^2$ we have the following decomposition
\begin{align}\label{eq:diff-ij}
    Z_i - \tauhat_n \Zdiese_j = \mu_i - \tauhat_n\mu_j + \overline{\xi}_{i, n} - \tauhat_n\Xibardiese_{j,n},
\end{align}
where $\mu_i = \theta_i - \bar{\theta}_n$ is defined in \eqref{eq:def-mu-xibar} and 
\begin{equation}
     \Xibar_{i,n} \eqdef \sigma_i \xi_i - \frac{1}{n} \sum_{i=1}^n \sigma_i \xi_i, \quad \Xibardiese_{i,n} \eqdef  \sdiese_i \xidiese_i - \frac{1}{n} \sum_{i=1}^n \sdiese_i \xidiese_i.
\end{equation}
We now proceed with upper bounding the normalized distance for the correct pair:
\begin{align}
        \frac{\|Z_i - \tauhat_n \Zdiese_i \|_2^2}{(1+\tauhat_n^2)s_i^2} &= \frac{\| \mu_i - \tauhat_n\mu_i\|_2^2 + 2(\mu_i - \tauhat_n\mu_i)^\top(\Xibar_{i,n} - \tauhat_n\Xibardiese_{i,n}) + \|\Xibar_{i,n} - \tauhat_n\Xibardiese_{i,n}\|_2^2 }{(1+\tauhat_n^2)s_i^2} \\ 
    &\le \frac{\| \mu_i - \tauhat_n\mu_i\|_2^2}{(1+\tauhat_n^2)s_i^2} + \frac{2\zeta_1 \|\mu_i - \tauhat_n\mu_i\|}{s_i\sqrt{1+\tauhat_n^2}} + d + \sqrt{d}\zeta_2,
\end{align}
where  $\zeta_1$ and $\zeta_2$ are defined as 
\begin{align}
    \zeta_1 \eqdef \sup_{\tauhat_n \in \mathcal{E}_{\tau}}\max_{i, j} \bigg| \frac{(\mu_i - \tauhat_n\mu_j)^\top(\Xibar_{i,n} - \tauhat_n\Xibardiese_{j,n})}{\| \mu_i - \tauhat_n\mu_j\| \sqrt{s_i^2 + \tauhat_n^2 s_j^2}}\bigg|, \qquad \zeta_2 \eqdef d^{-1/2} \sup_{\tauhat_n \in \mathcal{E}_{\tau}} \max_{i, j} \bigg| \frac{\|\Xibar_{i, n} - \tauhat_n\Xibardiese_{j, n}\|_2^2}{s_i^2 + \tauhat_n^2 s_j^2} - d\bigg|.
\end{align}
The lower bound on the distance between incorrect pairs ($i \neq j$) is constructed in a similar manner:
\begin{align}
    \frac{\|Z_i - \tauhat_n \Zdiese_j \|_2^2}{s_i^2 + \tauhat_n^2 s_j^2}  
    &\ge \frac{\| \mu_i - \tauhat_n\mu_j\|_2^2}{s_i^2 + \tauhat_n^2 s_j^2} - \frac{2\zeta_1 \|\mu_i - \tauhat_n\mu_j\|}{\sqrt{s_i^2+\tauhat_n^2s_j^2}} + d - \sqrt{d}\zeta_2.
\end{align}
Thus, the first event from \eqref{eq:inclusions} satisfies the following chain of inclusions 
\begin{align}
    &{\bigcup_{i=1}^n \bigcup_{j\not =i}\bigg\{\frac{\|Z_i - \tauhat_n \Zdiese_i \|_2^2}{(1+\tauhat_n^2)s_i^2}\ge\frac{\|Z_i - \tauhat_n \Zdiese_{j}\|_2^2}{s_i^2 + \tauhat_n^2 s_{j}^2}\bigg\} 
    \cap \Omega_{\tau}} \\ 
    &\subseteq 
    {\bigcup_{i=1}^n \bigcup_{j\not =i}\bigg\{\frac{\| \mu_i - \tauhat_n\mu_i\|_2^2}{(1+\tauhat_n^2) s_i^2}
    +\frac{2\zeta_1 \|\mu_i - \tauhat_n\mu_i\|}{s_i\sqrt{1+\tauhat_n^2}} + 2\sqrt{d}\zeta_2
    \ge \frac{\| \mu_i - \tauhat_n\mu_j\|_2^2}{s_i^2 + \tauhat_n^2 s_j^2} - \frac{2\zeta_1 \|\mu_i - \tauhat_n\mu_j\|}{\sqrt{s_i^2+\tauhat_n^2s_j^2}} \bigg\} \cap\Omega_{\tau}}\\ 
    &\subseteq{\bigcup_{i=1}^n \bigcup_{j\not =i}\bigg\{\tilde{\kappa}_{i,i}^2 + 2\zeta_1(\tilde{\kappa}_{i,i} + \tilde{\kappa}_{i,j}) + 2\sqrt{d}\zeta_2 \ge \tilde{\kappa}_{i,j}^2\bigg\}\cap\Omega_{\tau}} ,
\end{align}
where we define the proxy separation distance $\tilde{\kappa}_{i,j}$ as follows
\begin{align}\label{eq:proxy-kappa}
    \tilde{\kappa}_{i, j}^2 \eqdef \frac{\| \mu_i - \tauhat_n\mu_j\|_2^2}{s_i^2 + \tauhat_n^2 s_j^2}, \quad \forall (i, j) \in [n]^2, \quad \textup{ and } \quad \tilde{\kappa}_{\min} \eqdef \min_{i \neq j} \tilde{\kappa}_{i,j}.
\end{align}
On the event $\Omega_{\tau}$, the estimator $\tauhat_n$ is close to $\tau^*=1$, and thus $\tilde{\kappa}_{\min}$ is close to $\kappaZ$, hence its name.
Therefore, the probability of the event $\Omega$ satisfies
\begin{equation}\label{eq:Omega-A-ij}
    \Omega 
    \subseteq \brc{\bigcup_{i=1}^n \bigcup_{j\not =i} \cA_{i,j} \cap\Omega_{\tau}} \bigcup \Omega_{\tau}^{\sf c},
\end{equation}
where $\cA_{i,j} := \bigg\{\tilde{\kappa}_{i,i}^2 + 2\zeta_1(\tilde{\kappa}_{i,i} + \tilde{\kappa}_{i,j}) + 2\sqrt{d}\zeta_2 \ge \tilde{\kappa}_{i,j}^2\bigg\}$. 
In order to upper bound the right-hand side of \eqref{eq:Omega-A-ij}, we then 
show that $\cA_{i,j} \cap \Omega_{\tau}$ for all distinct pairs $(i,j)$ is a subset of  an event, the probability of which is smaller than $2\delta$. We first establish the relation between $\tilde{\kappa}_{\min}$ and $\kappaZ$. For any $i\neq j$, the reverse triangle inequality gives $\|\mu_i-\tauhat_n\mu_j\|\ge\|\mu_i-\mu_j\|-|1-\tauhat_n|\,\|\mu_j\|\ge\|\mu_i-\mu_j\|-\etau\|\mu_j\|$, while $\sqrt{s_i^2+\tauhat_n^2 s_j^2}\le(1+\etau)\sqrt{s_i^2+s_j^2}$ on $\Omega_\tau$; dividing, using $\sqrt{s_i^2+s_j^2}\ge s_j$, and minimizing over $i\neq j$ yields
\begin{align}\label{eq:kappa-min-kappa-z}
    \tilde{\kappa}_{\min} \geq \frac{1}{1+\varepsilon_{\tau}}\Big(\kappaZ - \varepsilon_{\tau}\max_{i\in[n]}\frac{\|\mu_i\|_2}{s_i}\Big).
\end{align}
Consequently, given $\kappaZ \ge (3+2\varepsilon_\tau)\,\varepsilon_{\tau} \max_{i\in[n]} \| \mu_i\|_2/s_i$ we get $\tilde{\kappa}_{\min} \ge 2\varepsilon_{\tau} \max_{i\in[n]} \| \mu_i\|_2/s_i$. On the other hand, we have the following upper bound on $\tilde{\kappa}_{i,i}$:
\begin{align}
    \tilde{\kappa}_{i,i} = \frac{|1-\tauhat_n| \| \mu_i \|_2}{s_i\sqrt{1+ \tauhat_n^2}} \le \frac{\varepsilon_{\tau}\|\mu_i\|_2}{s_i}.
\end{align}
Hence, we deduce that the following inequality holds for every pair $i\neq j$:
\begin{equation}\label{eq:kappa-ii-up}
    \max_{i\in[n]} \tilde{\kappa}_{i,i} \le \varepsilon_{\tau} \max_{i\in[n]} \frac{\| \mu_i\|_2}{s_i} 
    \leq \frac{1}{2} \tilde{\kappa}_{\min}
    \leq \frac{1}{2} \tilde{\kappa}_{i,j}.
\end{equation}
Direct application of \eqref{eq:kappa-ii-up} yields
\begin{align}
    \cA_{i,j}\cap \Omega_{\tau}
    &\subseteq \bigg\{ 3 \zeta_1 \tilde{\kappa}_{i,j} + 2\sqrt{d}\zeta_2 \ge \frac{3}{4}\tilde{\kappa}_{i,j}^2 \bigg\} \cap \Omega_{\tau}.
\end{align}
Notice also that on the event $\Omega_{\zeta_1} := \{\tilde{\kappa}_{\min} \ge 2\zeta_1\}$, the larger $\tilde{\kappa}_{i,j}$ the smaller the probability of the event from the previous display. Therefore, if we show that the probability of this event can be controlled by $2\delta$ for $\tilde{\kappa}_{\min}$ it would follow that it holds for arbitrary distinct pairs $i\neq j$. Hence, we arrive at the following inclusion:
\begin{align}
    \Omega \cap \Omega_{\zeta_1} &\subseteq \brs{\bigcup_{i=1}^n \bigcup_{j\not =i} \cA_{i,j} \cap\Omega_{\tau} \cap \Omega_{\zeta_1}} \cup \Omega_{\tau}^{\sf c}
    \subseteq \Big( \Big\{ 3\zeta_1 \tilde{\kappa}_{\min} + 2\sqrt{d}\zeta_2 \ge \frac{3}{4} \tilde{\kappa}_{\min}^2 \Big\}\cap \Omega_{\tau}\Big) \cup \Omega_{\tau}^{\sf c}.
\end{align}
Combining above inclusions we can bound the probability of the event $\Omega$ as follows:
\begin{align}
    \Prob(\Omega) &\le \Prob(\Omega_{\zeta_1}^{\sf c}) + \Prob(\Omega_{\tau}^{\sf c}) + \Prob\brr{ \brc{3\zeta_1 \tilde{\kappa}_{\min} + 2\sqrt{d}\zeta_2 \ge \frac{3}{4} \tilde{\kappa}_{\min}^2} \cap \Omega_{\tau}} \\ 
    &\le \Prob(\zeta_1 \ge \tilde{\kappa}_{\min}/2) + \Prob(\Omega_{\tau}^{\sf c})
     + \Prob(\zeta_1 \ge \tilde{\kappa}_{\min}/8) 
     + \Prob\brr{ \brc{\zeta_2\sqrt{d} \ge \frac{3}{16} \tilde{\kappa}_{\min}^2 } \cap \Omega_{\tau}} \\ 
    &\le 2\Prob(\zeta_1 \ge \tilde{\kappa}_{\min}/8) 
    + \Prob\brr{ \brc{\zeta_2\sqrt{d} \ge \frac{3}{16} \tilde{\kappa}_{\min}^2 } \cap \Omega_{\tau}}
     + \delta, \label{eq:final-proba-sum}
\end{align}
where in the last step we used the result of \Cref{thm:tau-affine} to bound the probability of the event $\Omega_{\tau}^{\sf c}$.
Now, using \Cref{lem:zeta1_bound} and \Cref{lem:zeta2_bound}, we find the threshold for $\tilde{\kappa}_{\min}$ such that both probabilities from \eqref{eq:final-proba-sum} are bounded by $\delta$. 
To keep the expressions concise, we notice that the dominating terms from \Cref{lem:zeta1_bound,lem:zeta2_bound} do not scale with $\etau$. 
Hence, we replace the terms that do depend on $\etau$ by the corresponding dominating terms. 
Formally, $\Prob(\zeta_1 \ge \tilde{\kappa}_{\min}/8) \le \delta/2$ holds as soon as $\tilde{\kappa}_{\min}/8$ exceeds the bound of \Cref{lem:zeta1_bound}, for which it suffices that
\begin{align}
    \tilde{\kappa}_{\min} \ge 8 \sqrt{2 \log\left(24 n^3 / \delta\right)},
    \qquad\text{i.e.}\quad \tilde{\kappa}_{\min}^2 \ge 128\,\log(24 n^3/\delta).
\end{align}
Similarly, $\Prob(\{\zeta_2\sqrt{d} \ge \frac{3}{16} \tilde{\kappa}_{\min}^2\}\cap\Omega_\tau) \le \delta$ holds as soon as $\frac{3}{16}\tilde{\kappa}_{\min}^2$ exceeds the bound of \Cref{lem:zeta2_bound}, for which it suffices that
\begin{align}
    \frac{3\tilde{\kappa}_{\min}^2}{16}\ge 2\Big(\sqrt{d \log(12 n^3 / \delta)} + \log(12 n^3 / \delta)\Big),
    \qquad\text{i.e.}\quad \tilde{\kappa}_{\min}^2 \ge \tfrac{32}{3}\big(\sqrt{d \log(12 n^3 / \delta)} + \log(12 n^3 / \delta)\big).
\end{align}
Combining these two bounds we get that as soon as  $\tilde{\kappa}_{\min}^2 \ge \tfrac{32}{3}\sqrt{d\log(12n^3/\delta)} + 128\log(12n^3/\delta)$, then the first two terms of \eqref{eq:final-proba-sum} are bounded by $\delta/2$ and $\delta$, respectively. Now, using the condition $\tilde\kappa_{\min}\ge 2\etau\max_i\|\mu_i\|_2/s_i$ from \eqref{eq:kappa-ii-up} along with the triangle inequality, we obtain that a sufficient condition 
\begin{align}\label{eq:kappa-min-bound}
    \tilde{\kappa}_{\min} \ge \sqrt{\tfrac{32}{3}}\,\big(d \log(12 n^3 / \delta)\big)^{1/4} + 8\sqrt{2 \log\left(24 n^3 / \delta\right)} + 2\varepsilon_{\tau}\max_{i\in[n]} \frac{\|\mu_i\|_2}{s_i},
\end{align}
ensures that the probability of recovering the wrong permutation $\pi^* \equiv \id$ is bounded by $3\delta$, \ie $\Prob(\Omega) \le 3\delta$. Recall the relation between $\tilde{\kappa}_{\min}$ and $\kappaZ$ from \eqref{eq:kappa-min-kappa-z}, then using \eqref{eq:kappa-min-bound} we arrive at the following condition on $\kappaZ$: 
\begin{align}
    \kappaZ \ge (1+\varepsilon_\tau)\Big( \sqrt{\tfrac{32}{3}}\big(d\log(12n^3/\delta)\big)^{1/4} + 8\sqrt{2\log(24n^3/\delta)} \Big) + (3+2\varepsilon_\tau)\,\varepsilon_{\tau}\max_{i\in[n]}\frac{\|\mu_i\|_2}{s_i}.
\end{align}
Using the fact that $\varepsilon_{\tau} \le 1/2$ one can verify that the latter reduces to
\begin{align}\label{eq:thm3-condition}
    \kappaZ \ge {} & 5\,\big(d \log(12 n^3 / \delta)\big)^{1/4} + 17 \,\sqrt{\log(24 n^3 / \delta)} + 4\,\etau\max_{i\in[n]}\frac{\|\mu_i\|_2}{s_i},
\end{align}
where we didn't make much effort to optimize the numerical constants. To further bound the last term from the previous display we use the result of \Cref{thm:tau-affine}, and apply the straightforward inequalities $\max_{i \in [n]} \| \mu_i\|_2 \le \lambda \|\bsigma\|_2$ and $\min_{i \in [n]}s_i \ge \sigma_{\min} \sqrt{1 - 1/n}$. Namely, with probability at least $1-\delta$, we have 
\begin{align}
    \varepsilon_{\tau} \max_{i\in[n]}\frac{\|\mu_i\|_2}{s_i} &\le 12 \sqrt{\frac{\alpha_{\sigma}^2\log(4/\delta)}{\lambda^2 + d}} \cdot \frac{\lambda\|\bsigma\|_2}{\sigma_{\min}\sqrt{1 - 1/n}}
    = 12\sqrt{\frac{\rho_{\sigma}\lambda^2 \log(4/\delta)}{(1- 1/n)(\lambda^2+d)}} \le 13 \sqrt{\rho_{\sigma}\log(4/\delta)},
\end{align}
where in the last inequality we used $\lambda^2 \le \lambda^2 + d$ and the assumption $n \ge 8$. Substituting back the upper bound from the previous display into \eqref{eq:thm3-condition} and using the union bound we get that under the required condition on $\kappaZ$ we have $\Prob(\Omega) \le 4\delta$, concluding the proof.

\end{proof}

\section{Proofs of the technical lemmas}

\subsection{Proof of Lemma \ref{lem:zeta1_bound}}\label{sec:proof-lem-zeta-1}

Define $g_i = u_{i,j}^\top \eta_i / \omega_i$ and $g^\prime_j = u_{i,j}^\top \eta^\prime_j / \omega_j$. 
Since $u_{i,j}$ is a unit vector and $\eta_i, \eta^\prime_i \sim N(0, \omega_i^2 I_d)$, the variables $g_i, g^\prime_i \sim N(0, 1)$ are independent scalar Gaussians for all $i\in[n]$. 
For a fixed pair $(i,j)$, we define
\begin{equation}
f_{ij}(a) := \frac{\omega_i g_i - a \omega_j g^\prime_j}{\sqrt{\omega_i^2 + a^2\omega_j^2}} =
 \frac{ u_{i,j}^\top \eta_i - a u_{i,j}^\top \eta^\prime_j}{\sqrt{\omega_i^2 + a^2\omega_j^2}}.
\end{equation}
Since $g_i$ and  $g^\prime_j$ are independent, for any fixed $a > 0$ we obtain $f_{ij}(a) \sim N(0,1)$. 
Thus, we need to bound $\Phi_1 = \sup_{a \in [1-\varepsilon, 1+\varepsilon]} \max_{i,j} \left|f_{ij}(a)\right|$. 
In order to bound this supremum, we construct a net with radius $r$ on the interval $[1-\epsilon,1+\epsilon]$.
We first bound the derivative of $f_{ij}(a)$ with respect to $a$. A direct calculation yields:
\begin{equation}
f'_{ij}(a) = -\frac{\omega_i^2 \omega_j g^\prime_j + a \omega_j^2 \omega_i g_i}{(\omega_i^2 + a^2\omega_j^2)^{3/2}}.
\end{equation}
The absolute value of the latter satisfies
\begin{equation}
|f'_{ij}(a)| \leq \frac{\omega_i^2 \omega_j |g^\prime_j| + a \omega_j^2 \omega_i |g_i|}{(\omega_i^2 + a^2\omega_j^2)^{3/2}} \leq  \frac{\omega_i \omega_j (\omega_i + a \omega_j)}{(\omega_i^2 + a^2\omega_j^2)^{3/2}}  \times
{M_g},
\end{equation}
where $M_g := \max \brc{|g_1|,|g^\prime_1|,\ldots,|g_n|,|g^\prime_n|}$.
Applying Cauchy-Schwartz we obtain
\begin{equation}
    \frac{\omega_i \omega_j (\omega_i + a \omega_j)}{(\omega_i^2 + a^2\omega_j^2)^{3/2}} \leq \frac{1}{\sqrt{2}a}.
\end{equation}
Since $a \in [1-\varepsilon, 1+\varepsilon]$ and $\varepsilon < 1/2$, we have $a > 1/2$. 
Thus, for all $i,j \in [n]$ and $a \in [1-\varepsilon, 1+\varepsilon]$
\begin{equation}
\sup_{a \in [1-\varepsilon, 1+\varepsilon]} |f'_{ij}(a)| \leq \sqrt{2}  M_g.
\end{equation}
On the other hand, $M_g$ is defined as the maximum (of absolute values) of standard Gaussian random variables.
Therefore, union bound and standard tail estimates imply that with probability $1-\delta/2$:
\begin{equation}
M_g \leq \sqrt{2 \log(8n/\delta)}.
\end{equation}
Thus,  the following  holds  with probability $1-\delta/2$:
\begin{equation}\label{eq:max-fprime-ij}
     \sup_{a \in [1-\varepsilon, 1+\varepsilon]} |f'_{ij}(a)| \leq 2\sqrt{\log(8n/\delta)}.
 \end{equation}
Let $\mathcal{N}$ be an $r$-net of the interval $[1-\varepsilon, 1+\varepsilon]$ with spacing $r = \frac{\varepsilon}{ n}$. 
The size of the net is $|\mathcal{N}| \leq \lceil \frac{2\varepsilon}{r} \rceil \leq 2n + 1 \leq 3n$.
For a fixed value of $a$, the total number of random variables $f_{ij}(a)$ on the grid is $N_{tot} = n^2 |\mathcal{N}| \leq 3n^3$.
Applying the union bound over all pairs $(i,j)$ and grid points $a_k \in \mathcal{N}$:
\begin{equation}
\mathbf{P}\left( \max_{i,j,k} |f_{ij}(a_k)| > t \right) \leq 2 N_{tot} e^{-t^2/2}.
\end{equation}
The latter implies that with probability $1-\delta/2$:
\begin{equation}\label{eq:max-f-ij}
\max_{i,j, a_k \in \mathcal{N}} |f_{ij}(a_k)| \leq \sqrt{2 \log\left(12 n^3 / \delta\right)} .
\end{equation}
Combining \eqref{eq:max-fprime-ij}, \eqref{eq:max-f-ij} and mean value theorem:
\begin{equation}
\Phi_1 \leq \max_{i,j,k} |f_{ij}(a_k)| + r \sup_{i,j,a} |f'_{ij}(a)|.
\end{equation}
Plugging in $r = \frac{\varepsilon}{n}$, we then obtain that with probability at least $1-\delta$:
\begin{equation}
\Phi_1 \leq \sqrt{2 \log\left(12 n^3 / \delta\right)} + \frac{2\varepsilon}{n} \sqrt{\log(8n/\delta)}.
\end{equation}
This concludes the proof.

\subsection{Proof of Lemma \ref{lem:zeta2_bound}}
\label{sec:proof-lem-zeta-2}
\begin{proof}
The proof follows a similar structure to \Cref{lem:zeta1_bound}.
We define the normalized random vectors $g_i = \eta_i / \omega_i$, $g^\prime_j = \eta^\prime_j / \omega_j$ and
\begin{align}
f_{ij}(a) &= \frac{\omega_i g_i - a \omega_j g^\prime_j}{\sqrt{\omega_i^2 + a^2\omega_j^2}}.
\end{align}
For any fixed $a$, $f_{ij}(a)$ is a linear combination of independent Gaussians normalized to have unit variance; thus $f_{ij}(a) \sim N(0, I_d)$. Our objective now is to upper bound the quantity $
Z = \sup_{a \in [1-\varepsilon, 1+\varepsilon]} \max_{i,j} \left|\|f_{ij}(a)\|^2- d\right|$. Again, as done in the proof of \Cref{lem:zeta1_bound}, we will bound $|\frac{d}{da} \|f_{ij}(a)\|^2|$ uniformly over $[1-\varepsilon, 1+\varepsilon]$ and $\|f_{ij}(a)\|^2$ on fixed $\epsilon$-net of points then employ a Lipschitz bound style argument.

Let us now bound the derivative of $\|f_{ij}(a)\|^2$ with respect to $a$. Using the chain rule we obtain
\begin{equation}
\Big|\frac{d}{da} \|f_{ij}(a)\|^2 \Big| 
= 2 \big|f_{ij}(a)^T f'_{ij}(a)\big|
\leq 2 \norm{f_{ij}(a)} \norm{f'_{ij}(a)}.
\end{equation}
We now give uniform upper bounds for $a\in [1-\varepsilon,1+\varepsilon]$ for each of the norms on the right-hand side. 
Using the triangle inequality and Cauchy-Schwarz:
\begin{equation}
\|f_{ij}(a)\| \leq \frac{\omega_i \|g_i\| + a \omega_j \|g^\prime_j\|}{\sqrt{\omega_i^2 + a^2\omega_j^2}} \leq \frac{\omega_i + a \omega_j}{\sqrt{\omega_i^2 + a^2\omega_j^2}} M_g \leq \sqrt{2} M_g,
\end{equation}
where $M_g = \max_{k} \brc{\max(\|g_k\|, \|g^\prime_k\|)}$.
Similarly,
\begin{align}
\norm{f'_{ij}(a) }
&= \norm{\frac{-(\omega_i^2 \omega_j g^\prime_j + a \omega_j^2 \omega_i g_i)}{(\omega_i^2 + a^2\omega_j^2)^{3/2}}}
\leq M_g \frac{\omega_i \omega_j (\omega_i + a \omega_j)}{(\omega_i^2 + a^2\omega_j^2)^{3/2}}.
\end{align}
As already stated in the proof of the previous lemma:
\begin{equation}
\frac{\omega_i \omega_j (\omega_i + a \omega_j)}{(\omega_i^2 + a^2\omega_j^2)^{3/2}} 
\leq  \frac{1}{\sqrt{2}a}.
\end{equation}
Since $\varepsilon < 1/2$, $a \geq 1/2$, implying $1/(\sqrt{2}a) \leq \sqrt{2}$. Therefore, $\|f'_{ij}(a)\| \leq \sqrt{2} M_g$.
Combining these results, the derivative of the squared norm is bounded by:
\begin{equation}
\left| \frac{d}{da} \|f_{ij}(a)\|^2 \right| \leq 4 M_g^2.
\end{equation}

Next, we bound $M_g^2$.
Recall, that each $g_k$ and $g^\prime_k$ are standard Gaussians, thus their norms are distributed as $\chi^2_d$. 
Using \Cref{lem:massart} and taking a union bound over $2n$ variables we get that with probability $1-\delta/2$:
\begin{equation}
M_g^2 \leq d + 2\sqrt{d \log(4n/\delta)} + 2\log(4n/\delta).
\end{equation}

Finally, we apply an $\varepsilon$-net argument. Let $\mathcal{N}$ be an $r$-net of $[1-\varepsilon, 1+\varepsilon]$ with $r = \varepsilon/n$. 
The net size is bounded by $3n$, resulting in $N_{tot} \leq 3n^3$ points to check for all pairs $(i,j)$.
For any fixed $a_k \in \mathcal{N}$, $\|f_{ij}(a_k)\|^2 \sim \chi^2_d$. Using the two-sided concentration bound $P(|X-d| \geq 2\sqrt{dx} + 2x) \leq 2e^{-x}$ from \Cref{lem:massart} and applying the union bound over $N_{tot}$ points with failure probability $\delta/2$:
\begin{equation}
\max_{i,j, k} \left| \|f_{ij}(a_k)\|^2 - d \right| \leq 2\sqrt{d \log(12 n^3 / \delta)} + 2\log(12 n^3 / \delta).
\end{equation}

Combining the discretization error and the Lipschitz bound:
\begin{equation}
Z \leq \max_{i,j,k} \left| \|f_{ij}(a_k)\|^2 - d \right| + r \sup_{i,j,a} \left| \frac{d}{da} \|f_{ij}(a)\|^2 \right|.
\end{equation}
Substituting $r = \varepsilon/n$ and the derived bounds, with probability at least $1-\delta$:
\begin{equation}
Z \leq 2\sqrt{d \log(12 n^3 / \delta)} + 2\log(12 n^3 / \delta) + \frac{4\varepsilon}{n} \left( d + 2\sqrt{d \log(4n/\delta)} + 2\log(4n/\delta) \right).
\end{equation}
\end{proof}

\section{Experiments on Real-World Translation Data}
\label{subsec:real_lang}

We validate our method on cross-encoder sentence matching task. This and similar tasks arise in real-world data pipelines, where embeddings or features from different models are compared or analyzed, e.g. multi-vendor aggregation in RAG systems, asymmetric query/document encoders, etc.
Concretely, our goal is to match English sentences from OPUS-100 parallel corpora \citep{opus} to their corresponding translations across 7 languages (fr, hy, my, kk, ka, ky, ja). Note that the ground-truth is available, as OPUS-100 is sentence-level parallel. 
English-side sentences are encoded using \textit{all-MiniLM-L6-v2} \citep{reimers-2019-sentence-bert}, while for target-language side we use completely different \textit{paraphrase-multilingual-MiniLM-L12-v2} encoder \citep{reimers-2020-multilingual-sentence-bert}.
Results are shown in \Cref{fig:rw_lang}.

Two encoders consistently differ in centered scale by a factor of $\approx 3$, with $\tauhat$ tightly concentrated across seeds in line with the rate predicted by \Cref{thm:tau-affine}.
On well-resourced pairs (en-fr) both methods recover the matching with $\approx 96\%$ accuracy and the gap is negligible, directly illustrating our \textit{standardization cost} caveat.
On low-resource pairs, where the multilingual encoder produces noisier embeddings, naive LSL accuracy degrades to as low as 23\% (en-ky) while Affine LSL maintains a strict advantage in every pair.

\Cref{fig:rw_lang_pca} makes this geometric: the raw embedding clouds occupy visibly different regions of PC space, while standardization collapses them onto a shared geometry.

\begin{figure}
    \centering
    \includegraphics[width=0.48\textwidth]{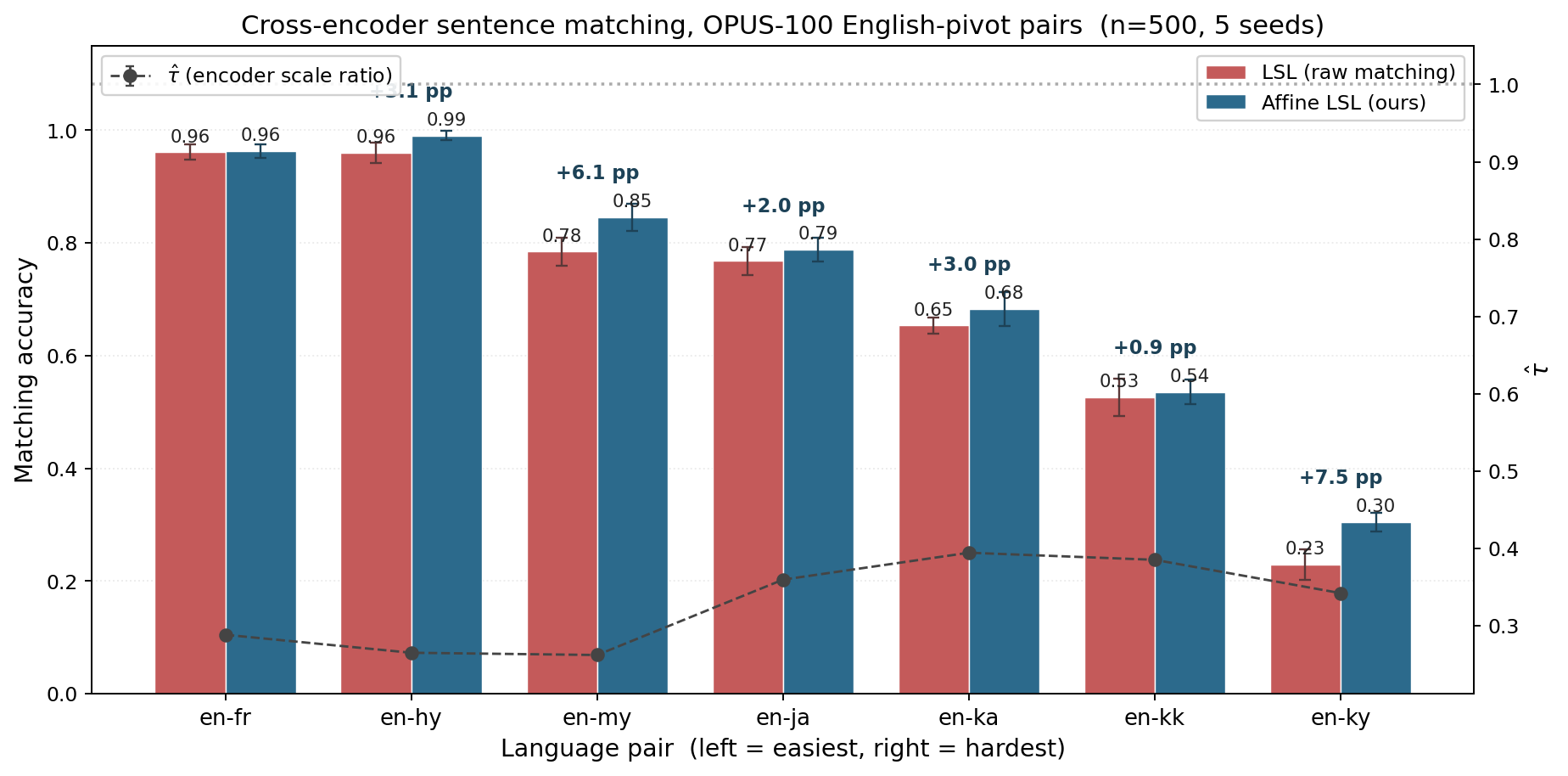}
    \caption{Accuracies of vanilla LSL and Affine LSL (ours) on sentence matching task. For each pair we match 500 English sentences to their translations in 7 pivot languages. As one can see, our approach outperforms vanilla LSL on all, especially harder pairs. Also overlayed is predicted $\tauhat$ for each pair. All experiments are repeated and averaged across 5 independent trials.}
    \label{fig:rw_lang}
\end{figure}

\begin{figure}
    \centering
    \includegraphics[width=0.48\textwidth]{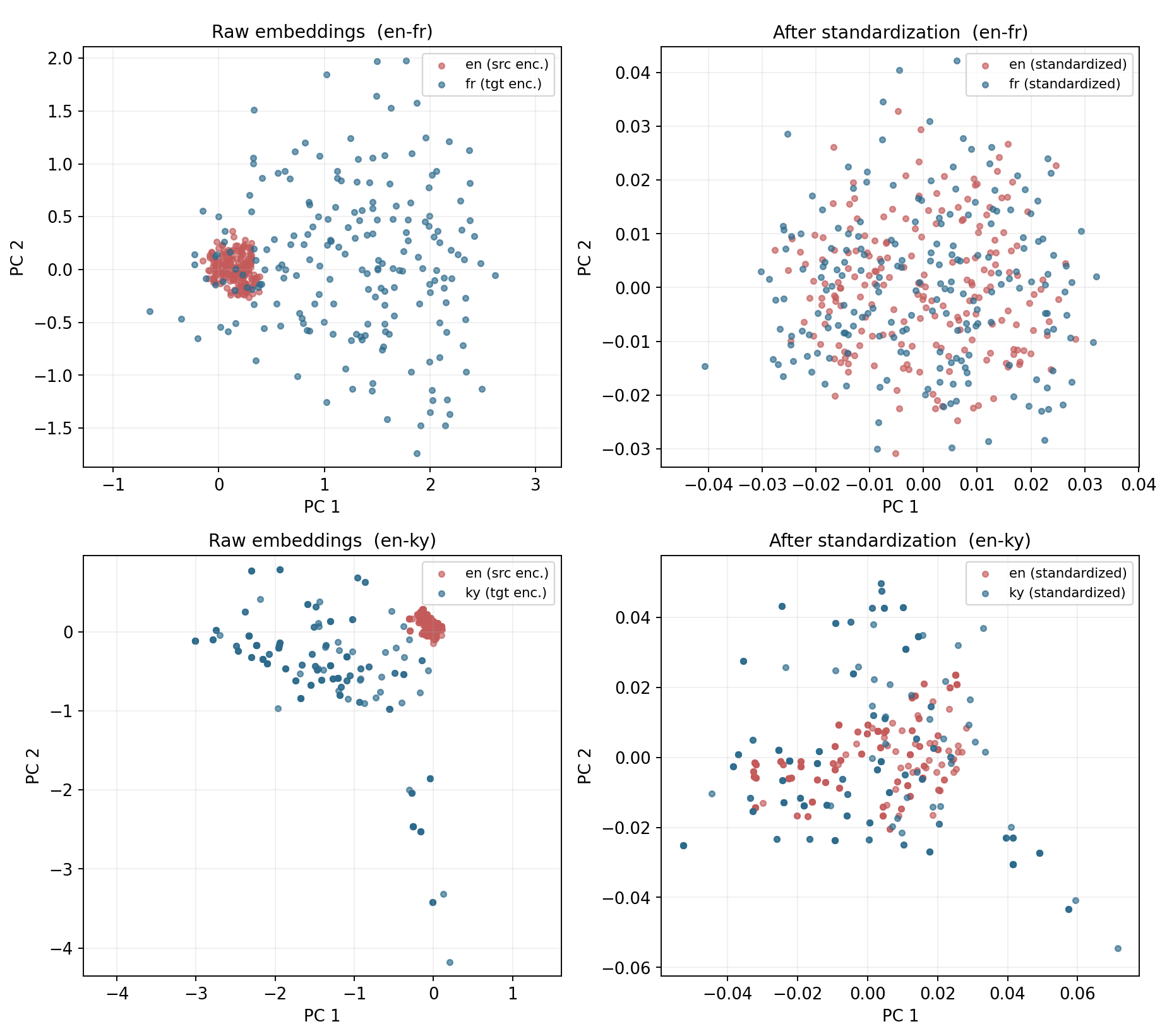}
    \caption{PCA visualization of OPUS-100 en-fr (easy) and en-ky (hard) sentence embeddings. Note that here standardization changes the geometry of the space to better suit the matching problem down the line.}
    \label{fig:rw_lang_pca}
\end{figure}